\documentclass[10pt]{amsart} 
\usepackage{latexsym} 
\usepackage{amssymb} 
\usepackage{amscd} 
\usepackage{epsf}
\usepackage{amsmath,amssymb,graphicx,amsthm,mathrsfs,verbatim}
\usepackage{pstricks}

\begin{document}

\newtheorem{Thm}{Theorem}[section] 
\newtheorem{TitleThm}[Thm]{} 
\newtheorem{Corollary}[Thm]{Corollary} 
\newtheorem{Proposition}[Thm]{Proposition} 
\newtheorem{Lemma}[Thm]{Lemma} 
\newtheorem{Conjecture}[Thm]{Conjecture} \theoremstyle{definition} 
\newtheorem{Definition}[Thm]{Definition} \theoremstyle{definition} 
\newtheorem{Example}[Thm]{Example} 
\newtheorem{TitleExample}[Thm]{} 
\newtheorem{Remark}[Thm]{Remark} 
\newtheorem{SimpRemark}{Remark} 
\renewcommand{\theSimpRemark}{}

\numberwithin{equation}{section}

\newcommand{\C}{{\mathbb C}}
\newcommand{\bbH}{{\mathbb H}}
\newcommand{\R}{{\mathbb R}} 
\newcommand{\Z}{{\mathbb Z}}
\newcommand{\RR}{{\mathbb{R}}}

\newcommand{\flushpar}{\par \noindent}
\newcommand{\codim}{{\rm codim}\,}
\newcommand{\Diff}{{\rm Diff}\,} 
\newcommand{\wt}{{\rm wt}\,}
\newcommand{\pd}[2]{\dfrac{\partial#1}{\partial#2}}
\def \dim {{\rm dim}\,}
\newcommand{\bdyM}{\partial M} 
\def \mod {{\rm mod}\;}

\newcommand{\proj}{{\rm proj}} 

\def \ba {\mathbf {a}}
\def \bb {\mathbf {b}}
\def \bB {\mathbf {B}}
\def \bc {\mathbf {c}}
\def \be {\mathbf {e}}
\def \bone {\boldsymbol {1}}
\def \bh {\mathbf {h}}
\def \bk {\mathbf {k}}
\def \bm {\mathbf {m}}
\def \bn {\mathbf {n}}
\def \bN {\mathbf {N}}
\def \bp {\mathbf {p}}
\def \bt {\mathbf {t}}
\def \bT {\mathbf {T}}
\def \bu {\mathbf {u}}
\def \bv {\mathbf {v}}
\def \bx {\mathbf {x}}
\def \bw {\mathbf {w}}
\def \b1 {\mathbf {1}}
\def \bga {\boldsymbol \alpha}
\def \bgb {\boldsymbol \beta}
\def \bgg {\boldsymbol \gamma}
\def \bge {\boldsymbol \epsilon}

\def \itc {\text{\it c}}
\def \iti {\text{\it i}}
\def \itj {\text{\it j}}
\def \itm {\text{\it m}}
\def \itM {\text{\it M}}  
\def \ithn {\text{\it hn}}
\def \itt {\text{\it t}}

\def \cA {\mathcal{A}}
\def \cB {\mathcal{B}}
\def \cC {\mathcal{C}}
\def \cD {\mathcal{D}}
\def \cE {\mathcal{E}}
\def \cF {\mathcal{F}}
\def \cG {\mathcal{G}}
\def \cH {\mathcal{H}}
\def \cK {\mathcal{K}}
\def \cL {\mathcal{L}}
\def \cM {\mathcal{M}}
\def \cN {\mathcal{N}}
\def \cO {\mathcal{O}}
\def \cP {\mathcal{P}}
\def \cQ {\mathcal{Q}}
\def \cR {\mathcal{R}}
\def \cS {\mathcal{S}}
\def \cT {\mathcal{T}}
\def \cU {\mathcal{U}}
\def \cV {\mathcal{V}}
\def \cW {\mathcal{W}}
\def \cX {\mathcal{X}}
\def \cY {\mathcal{Y}}
\def \cZ {\mathcal{Z}}

\def \ga {\alpha}
\def \gb {\beta}
\def \gg {\gamma}
\def \gd {\delta}
\def \ge {\epsilon}
\def \gevar {\varepsilon}
\def \gk {\kappa}
\def \gl {\lambda}
\def \gs {\sigma}
\def \gt {\tau}
\def \gw {\omega}
\def \gz {\zeta}
\def \gG {\Gamma}
\def \gD {\Delta}
\def \gL {\Lambda}
\def \gS {\Sigma}
\def \gW {\Omega}

\def \dim {{\rm dim}\,}
\def \mod {{\rm mod}\;}

\title[Lorentzian Geodesic Flows]
{Lorentzian Geodesic Flows and Interpolation between Hypersurfaces in Euclidean Spaces}
\author[J. Damon]{James Damon} 
\thanks{Partially supported by DARPA grant HR0011-05-1-0057 and 
National Science Foundation grant DMS-1105470}
\address{Department of Mathematics \\
University of North Carolina \\
Chapel Hill, NC 27599-3250  \\
USA}


\begin{abstract}
We consider geodesic flows between hypersurfaces in $\R^n$.  However, rather than consider using geodesics in $\R^n$, which are straight lines, we consider an induced flow using geodesics between the tangent spaces of the hypersurfaces viewed as affine hyperplanes.  For naturality, we want the geodesic flow to be invariant under rigid 
transformations and homotheties.  Consequently, we do not use the dual projective space, as the geodesic flow in this space is not preserved under translations.  Instead we give an alternate approach using a Lorentzian space, which is semi-Riemannian with a metric of index $1$. \par
For this space for points corresponding to affine hyperplanes in $\R^n$, we give a formula for the geodesic between two such points.  As a consequence, we show the geodesic flow is preserved by rigid transformations and homotheties of $\R^n$.  Furthermore, we give a criterion that a vector field in a smoothly varying family of hyperplanes along a curve yields a Lorentzian parallel vector field for the corresponding curve in the Lorentzian space.  As a result this provides a method to extend an orthogonal frame in one affine hyperplane to a smoothly \lq\lq Lorentzian varying\rq\rq family of 
orthogonal frames in a family of affine hyperplanes along a smooth curve, as well as a interpolating between two such frames with a smooth \lq\lq minimally Lorentzian varying\rq\rq family of orthogonal frames. \par
We further give sufficient conditions that the Lorentzian flow from a hypersurface is nonsingular and that the resulting corresponding flow in $\R^n$ is nonsingular.  This is illustrated for surfaces in $\R^3$. 
\end{abstract}


\maketitle
\vspace{3ex}
\centerline{PRELIMINARY VERSION}
\vspace{3ex}

\section*{\bf Introduction} 
\par

We consider the problem of constructing a natural diffeomorphic flow 
between hypersurfaces $M_0$ and $M_1$ of $\R^n$ which is in some sense 
both \lq\lq natural\rq\rq and \lq\lq geodesic\rq\rq viewed in some 
appropriate space (as in figure   ).  
\par
\begin{figure}[ht]
\centerline{
\includegraphics[width=10cm]{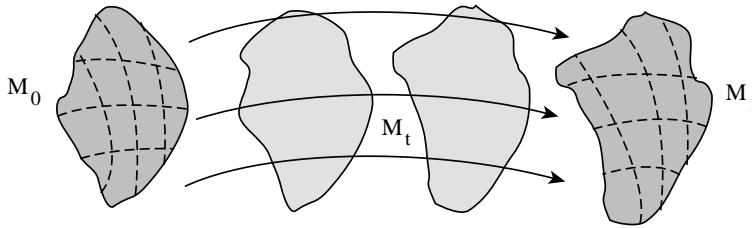}}
\caption{\label{fig.3}  Diffeomorphic Flow between hypersurfaces of 
Euclidean space induced by a \lq\lq Geodesic Flow\rq\rq in an associated 
space}
\end{figure} 
 \par
There are several approaches to this question.  One is from the 
perspective  
of a Riemannian metric on the group of diffeomorphisms of $\R^n$.  If the 
smooth hypersurfaces $M_i$ bound compact regions $\gW_i$ , then the 
group of diffeomorphisms $\Diff(\R^n)$ acts on such regions $\gW_i$  and 
their boundaries.  Then,  if $\varphi_t, 1\leq t \leq 1,$ is a geodesic in 
$\Diff(\R^n)$ beginning at the identity, then $\varphi_t(\gW)$ (or 
$\varphi_t(M_i)$) provides a path interpolating between $\gW_0 = 
\varphi_0(\gW) = \gW$ and $\gW_1 = \varphi_1(\gW)$.  Then, the geodesic 
equations can be computed and numerically solved to construct the flow 
$\varphi_t$.  This is the method developed by Younes, Trouve, Glaunes 
\cite{Tr}, \cite{YTG}, \cite{BMTY}, \cite{YTG2}, and Mumford, Michor 
\cite{MM}, \cite{MM2} etc.  \par
An alternate approach which we consider in this paper requires that we 
are given a correspondence between $M_0$ and $M_1$, defined by a 
diffeomorphism $\chi : M_0 \to M_1$, which need not be the restriction of 
a global diffeomorphism of $\R^n$ (and the $M_i$ may have boundaries).  
Then, if we map $M_0$ and $M_1$ to submanifolds of a natural ambient 
space $\gL$, we can seek a \lq\lq geodesic flow\rq\rq between $M_0$ and 
$M_1$, viewed as submanifolds of $\gL$, sending $x$ to $\varphi(x)$ along 
a geodesic.  Then, we use this geodesic flow to define a flow between 
$M_0$ and $M_1$ back in $\R^n$.  \par
The simplest example of this is the \lq\lq radial flow\rq\rq from $M_0$ 
using the vector field $U$ on $M_0$ defined by $U(x) = \varphi(x) - x$.  
Then, the radial flow is the geodesic flow in $\R^n$ defined by 
$\varphi_t(x) = x + t\cdot U(x)$.  The analysis of the nonsingularity of the 
radial flow is given in \cite{D1} in the more general context of 
\lq\lq skeletal structures\rq\rq.  
This includes the case where $M_1$ is a \lq\lq generalized offset 
surface\rq\rq  of $M_0$ via the generalized offset vector field $U$.  \par
\begin{figure}[ht]
\begin{center} 
\includegraphics[width=5cm]{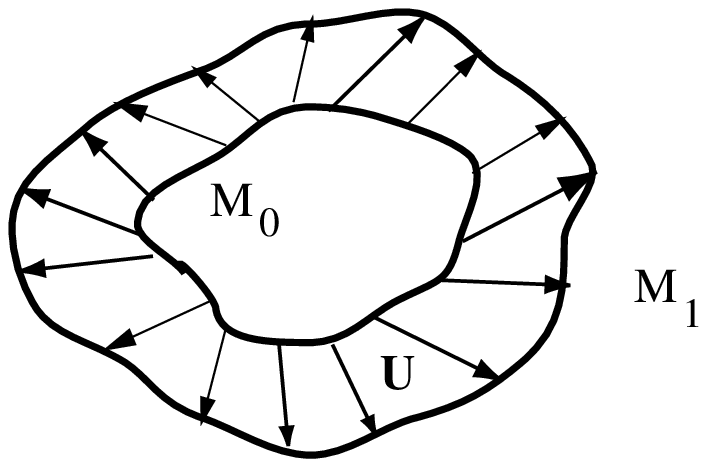}  \hspace{.25in} 
\includegraphics[width=5cm]{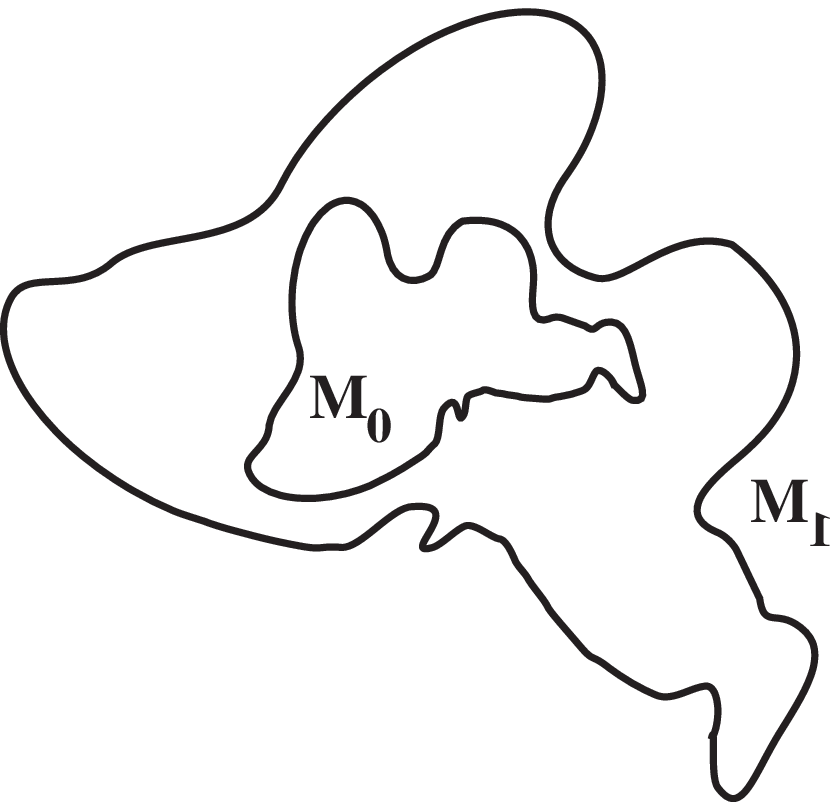}
\end{center}
\hspace{0.4in} (a) \hspace{1.8in} (b) 
 
\caption{\label{fig.2} a) Hypersurface $M_0$ and radial vector field $U$ 
define a generalized offset surface $M_1$ obtained from a radial flow of 
the skeletal structure $(M_0, U)$.  This gives a nonsingular \lq\lq Geodesic 
Flow\rq\rq in $\R^n$.  In b) there is no nonsingular geodesic flow in $\R^n$ 
from $M_0$ to $M_1$; however, there is a nonsingular Lorentzian geodesic
 flow from $M_0$ to $M_1$ (see \S \ref{Sec:ind.flow})}  
\end{figure} 
\par

In this paper, we give an alternate approach to interpolation via a geodesic 
flow between hypersurfaces with a given correspondence.  While the radial 
flow views each hypersurface as a collection of points, we will instead view 
each as defined by their collection of tangent spaces.  This leads to consideration 
of geodesic flows between \lq\lq dual varieties\rq\rq.  The dual 
varieties traditionally lie in the \lq\lq dual projective space\rq\rq.  However, 
 the geodesic flow induced on the dual projective space with its natural 
Riemannian metric does not have certain natural properties that are desirable, 
such as invariance under translation.  Instead, we shall define in \S 
\ref{Sec:Lorentz.map} a \lq Lorentzian map\rq\rq to a hypersurface 
$\cT^n$ in the Lorentzian space $\gL^{n+1}$ defined by their tangent 
spaces as affine hyperplanes in $\R^{n+1}$.  So instead of representing 
hypersurfaces in terms of \lq\lq dual varieties\rq\rq, we instead 
represent them as subspaces of $\gL^{n+1}$, which is the subspace of 
points in Minkowski space $\R^{n+2, 1}$ of Lorentzian norm $1$.  Then, we 
use the geodesic flow for the Lorentzian metric on $\gL^{n+1}$, and then 
transform that geodesic flow back to a flow between the original 
manifolds in $\R^n$.  \par
To do this we determine in \S \ref{Sec:geo.deS} the explicit form for the 
Lorentzian geodesics in $\gL^{n+1}$ between points in $\cT^n$.  We show 
these geodesics lie in $\cT^n$ and show in \S \ref{Sec:spec.cas}  that 
these geodesics are invariant under the extended Poincare group.  Given a 
Lorentzian geodesic between two points in $\cT^n$, there corresponds a 
smooth family of hyperplanes $\Pi_t$.  \par 
We further give in \S \ref{Sec:par.frms} a criterion for \lq\lq Lorentzian parallel 
vector fields\rq\rq\, in a family of hyperplanes $\Pi_t$ along a curve 
$\gg(t)$ in $\R^n$, and then determine the Lorentzian parallel vector fields over 
a Lorentzian geodesic corresponding to vector fields with values in $\Pi_t$.  Using 
this, we determine for an orthonormal frame $\{e_{i, 0}\}$ in $\Pi_0$,  a 
smooth family of orthonormal frames $\{e_{i, t}\}$ in $\Pi_t$ which 
correspond to a Lorentzian parallel family of frames along the Lorentzian 
geodesic.  Using this we further determine a method for interpolating 
between orthonormal frames $\{e_{i, 0}\}$ in $\Pi_0$ and $\{e_{i, 1}\}$ in 
$\Pi_1$.  \par
In \S\S \ref{Sec:dual.var} and \ref{Sec:smoot.lev} we relate the properties 
of hypersurfaces $\tilde M$ of $\cT^n$ with corresponding properties of 
the envelopes formed from the planes defined by $\tilde M$.  In \S 
\ref{Sec:dual.var} we give a diffeomorphism between the Lorentzian space 
$\cT^n$ with the dual projective space $\R P^{n\, \vee}$, which is a 
Riemannian manifold.  The classification of generic Legendrian 
singularities in $\R P^{n\, \vee}$ gives the form of the singular points of 
images and this is used to give criteria for the lifting to a hypersurface in 
$\R^{n+1}$ as the envelope of the family of corresponding hyperplanes.  
\par
Finally, in section \ref{Sec:ind.flow}  we give in Theorem \ref{Thm6.1} the 
existence and continuous dependence of the corresponding \lq\lq 
Lorentzian geodesic flow\rq\rq\, between two hypersurfaces $M_0$ and 
$M_1$ in $\R^{n+1}$  and in Theorem \ref{Thm9.2} we give a sufficient 
condition for the flow to be nonsingular.  As a special case we consider in 
\S \ref{Sec:surf.case} the results for surfaces in $\R^3$. 
\par

 \vspace{5ex}
\newpage
\centerline{CONTENTS}
 \vspace{2ex}
\begin{enumerate}

\item	Overview
\vspace{2ex} \par
\item  Semi-Riemannian Manifolds and Lorentzian Manifolds 
\par 
\vspace{2ex} 
\par
\item	Definition of the Lorentzian Map
\par
\vspace{2ex} 
\par
\item	Lorentzian Geodesic Flow on $\gL^{n+1}$
\par 
\vspace{2ex} 
\par
\item	Invariance of Lorentzian Geodesic Flow and Special Cases 
\par
\vspace{2ex}
\par
\item	Families of Lorentzian Parallel Frames on Lorentzian Geodesic 
Flows
\par
\vspace{2ex}
\par
\item	Dual Varieties and Singular Lorentzian Manifolds
\par 
\vspace{2ex}
\par
\item	Sufficient Condition for Smoothness of Envelopes
\vspace{2ex}
\par
\item	Induced Geodesic Flow between Hypersurfaces
\par 
\vspace{2ex}
\par
\item	 Results for the Case of Surfaces in $\R^3$ 
\vspace{2ex}  
\par

\vspace{2ex}
\end{enumerate}

\vspace{3ex}

 \newpage
\section{\bf Overview}
\label{Sec:overview} 
\par
As mentioned in the introduction there are two main methods for 
deforming 
one given hypersurface $M_0 \subset \R^n$ to another $M_1$.  One is to 
find a path $\psi_t$ in $G$, which is some specified a group of 
diffeomorphisms of $\R^n$, from the identity so that $\psi_1(M_0) = M_1$ 
(and $\psi_0(M_0) = M_0$).  \par  
Another approach involves constructing a geometric flow between $M_0$ 
and 
$M_1$.  Several flows such as curvature flows do not provide a flow to a 
specific hypersurface such as $M_1$.  An alternate approach which we 
shall 
use will assume that we have a correspondence given by a diffeomorphism 
$\chi : M_0 \to M_1$ and construct a \lq\lq geodesic flow\rq\rq which at 
time $t = 1$ gives $\chi$.  The geodesic flow will be defined using an 
associated space $\cY$.  We shall consider natural maps 
$\varphi_i : M_i \to \cY$ , where $\cY$ is a distinguished space which 
reflects certain geometric properties of the $M_i$. 
\begin{eqnarray}
\label{tag1.1}
{M_0} & \overset{\, \,\, \varphi_0}{\longrightarrow} &  \cY \notag\\
   {\chi \downarrow}  &    \underset{\, \,\, \varphi_1}{\nearrow}    &        
\\
 {M_1}  &       &   \notag  
\end{eqnarray}

 \par
\begin{Definition}
\label{Def1.0}
Given smooth maps $\varphi_i : M_i \to \cY$ and a diffeomorphism $\chi : 
M_0 \to M_1$  A {\em geodesic flow} between the maps $\varphi_i$ is a 
smooth map $\tilde \psi_t :  M_0 \times [0, 1] \to \cY$ such that for any 
$x 
\in M_0$, $\tilde \psi_t(x) : [0, 1] \to \cY$ is a geodesic from 
$\varphi_0(x)$ to $\varphi_1 \circ \chi (x)$ 
\end{Definition} \par
\begin{SimpRemark}
We shall also refer to the geodesic flow as being between the $\tilde M_i 
= 
\varphi_i(M_i)$. However, we note that it is possible for more than one 
$x_i 
\in M_0$ to map to the same point in $y \in \cY$, however, the geodesic 
flow 
from $y$ can differ for each point $x_i$.  
\end{SimpRemark}
Then, we will complement this with a method for finding the 
corresponding 
flow $\psi_t$ between $M_0$ and $M_1$ such that $\varphi_t \circ \psi_t 
= 
\tilde \psi_t$, where $\varphi_t: \psi_t(M_0) \to \tilde \psi_t(M_0)$.  We 
furthermore want this flow to satisfy certain properties.  A main property 
is that the flow construction is invariant under the action of the extended 
Poincare group formed from rigid transformations and homotheties (scalar 
multiplication).  
By this we mean: if $M_0^{\prime} = A(M_0)$ and $M_1^{\prime} = A(M_1)$ 
are transforms of $M_0$ and $M_1$ by a transformation $A$ formed from 
the composition of a rigid transformation and homothety, and $M_t$ is the 
flow between $M_0$ and $M_1$, then $A(M_t)$ gives 
the flow between $M_0^{\prime}$ and $M_1^{\prime}$.  \par
We are specifically interested in a \lq\lq geodesic flow\rq\rq which will 
be a flow defined using the tangent bundles $TM_0$ to $TM_1$ so that we 
specifically control the flow of the tangent spaces.  At first, an apparent 
natural choice is the dual projective space $\R P^{n \vee}$.  Via the 
tangent bundle of a hypersurface $M \subset \R^n\, (\subset \R P^n)$ there 
is the natural map $\gd : M \to  \R P^{n \vee}$, sending $x \mapsto T_xM$.   
The natural 
Riemannian structure on the real projective space $\R P^{n \vee}$ is 
induced from $S^n$ via the natural covering map $S^n \to \R P^n$, so that 
geodesics of $S^n$ map to geodesics on $\R P^{n \vee}$.  However, simple 
examples show that the induced geodesic slow on $\R P^{n \vee}$ is not 
invariant under translation in $\R^n$.  For example, this Riemannian 
geodesic flow between the hyperplanes given by $\bn \cdot \bx = c_0$ and 
$\bn \cdot \bx = c_1$ is given by $\bn \cdot \bx = c_t$,  where $c_t = 
\tan (t\arctan (c_1) + (1-t)\arctan (c_0))$.  It is easily seen that if we 
translate the two planes by adding a fixed amount $d$ to each $c_i$, then 
the corresponding formula does not give the translation of the first.  \par 
We will use an alternate space for $\cY$, namely, the Lorentzian space 
$\gL^{n+1}$ which is a Lorentzian subspace of Minkowski space 
$\R^{n+2, 1}$.   In fact the images will be in an $n$-dimensional 
submanifold $\cT^n \subset \gL^{n+1}$.  On $\gL^{n+1}$ it is classical that 
the geodesics are intersections with planes through the origin in 
$\R^{n+2, 1}$.  This allows a simple description of the geodesic flow on 
$\gL^{n+1}$.  We transfer this flow to a flow on $\R^n$ using an inverse 
envelope construction, which reduces to solving systems of linear 
equations.  We will give conditions for the smoothness of the inverse 
construction which uses knowledge of the generic Legendrian 
singularities.  \par
We shall furthermore see that the construction is invariant under the 
action of rigid transformations and homotheties.   In addition, uniform 
translations and homotheties will be geodesic flows, and a \lq\lq pseudo 
rotation\rq\rq which is a variant of uniform rotation is also a geodesic 
flow.  

\vspace{2ex}

\section{\bf Semi-Riemannian Manifolds and Lorentzian Manifolds}
\label{Sec:Rie.Lor} 
\par
A {\it Semi-Riemannian manifold} $M$ is a smooth manifold $M$, with a  
nondegenerate bilinear form $< \cdot, \cdot >_x$ on the tangent space 
$T_xM$, for eaxh $x \in M$ which smoothly varies with $x$.  We do not 
require that $< \cdot, \cdot >_x$ be positive definite.  We denote the index 
of $< \cdot, \cdot >_{(x)}$ by $\nu$.  In the case that $\nu = 1$, $M$ is 
referred to as a Lorentzian manifold.  \par
A basic example is Minkowski space which (for our purposes) is $\R^{n+2}$ 
with bilinear form defined for $v = (v_1, \dots , v_{n+2})$ and  
$w = (w_1, \dots , v_{n+2})$
$$  < v, w >_L  \quad = \quad \sum_{i = 1}^{n+1} v_i\cdot w_i -   
v_{n+2}\cdot w_{n+2}  $$
There are a number of different notations for this Minkowski space.  We shall 
use $\R^{n+2, 1}$.  We shall also use the notation $<\cdot, \cdot >_L$ for the 
Lorentzian inner product on $\R^{n+2, 1}$.  \par
A submanifold $N$ of a semi-Riemannian manifold $M$ is a 
semi-Riemannian submanifold if for each $x \in N$, the restriction of 
$< \cdot, \cdot >_{(x)}$ to $T_xN$ is nondegenerate.  
There are several important submanifolds of $\R^{n+2, 1}$.  One such is 
the Lorentzian submanifold
$$  \gL^{n+1} \,\,  = \,\, \{ (v_1, \dots , v_{n+2}) \in \R^{n+2, 1} : \sum_{i = 
1}^{n+1} v_i^2 \, -\,  v_{n+2}^2 \, = \,1\},  $$  
which is called de Sitter space (see Fig. \ref{fig.4}). 
A second important one is {\it hyperbolic space} $\bbH^{n+1}$ defined by
$$ \bbH^{n+1} \,\,  = \,\, \{ (v_1, \dots , v_{n+2}) \in \R^{n+2, 1} : 
\sum_{i = 1}^{n+1} v_i^2 \, -\,  v_{n+2}^2 \,  = \,  -1\ \,\, \mbox{and } \, 
v_{n+2} > 0\}.  $$
\par
\begin{figure}[ht]
\centerline{\includegraphics[width=5cm]{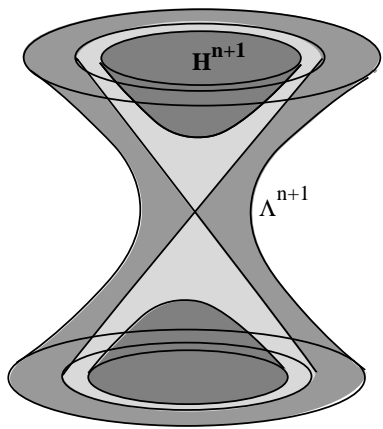}}
\caption{\label{fig.4}  In Minkowski space $\R^{n+2, 1}$, there is the 
Lorentzian hypersurface $\gL^{n+1}$ and the model for hyperbolic space 
$\bbH^{n+1}$.  Also shown is the \lq\lq light cone\rq\rq.}
\end{figure} 
 \par

By contrast the restriction of $< \cdot, \cdot >_L$ to $\bbH^{n+1}$ is a 
Riemannian metric of constant negative curvature $-1$.  There is natural 
duality between codimension $1$ submanifolds of $\bbH^ {n+1} $ obtained as the 
intersection of $\bbH^ {n+1} $ with a \lq\lq time-like\rq\rq hyperplane $\Pi$ 
through $0$ (containing a \lq\lq time-like\rq\rq vector $z$ with $<z, z>_L 
\, < 0$) paired with the points $ \pm z^{\prime} \in \gL^{n+1}$ given where 
$z^{\prime}$ lies on a line through the origin which is the Lorentzian 
orthogonal complement to $\Pi$. \par
Many of the results which hold for Riemannian manifolds also hold for a 
Semi-
Riemannian manifold $M$.  
\begin{TitleThm}[Basic properties of Semi-Riemannian Manifolds (see 
\cite{ON}] 
\label{Thm2.1} \hfill \par

For a Semi-Riemannian manifold $M$, there are the following properties 
analogous to those for Riemannian manifolds: \par
\begin{enumerate}
\item  Smooth Curves on $M$ have lengths defined using 
$| < \cdot, \cdot > |$.  
\item  There is a unique connection which satisfies the usual properties 
of a Riemannian Levi-Civita connection.
\item  Geodesics are defined locally from any point $x \in M$ and with any 
initial velocity $v \in T_xM$.  They are critical curves for the length 
functional, and they have constant speed.
\item If $N$ is a semi-Riemannian submanifold of $M$, then a constant 
speed curve $\gg(t)$ in $N$ is a geodesic in $N$ if the acceleration 
$\gg^{\prime\prime}(t)$ is normal to $N$ (with respect to the 
semi-Riemannian metric) at all points of $\gg(t)$.
\item  Any point $x \in M$ has a \lq\lq convex neighborhood\rq\rq\, $W$, 
which has the property that any two points in $W$ are joined by a unique 
geodesic in the neighborhood.
\item  If $\gg(t)$ is a geodesic joining $x_0 = \gg(0)$ and $x_1 = \gg(1)$ 
and $x_0$ and $x_1$ are not conjugate along $\gg(t)$, then given a 
neighborhood $W$ of $\gg(t)$, there are neighborhoods of $W_0$ of $x_0$ 
and $W_1$ of $x_1$ so that if $x_0^{\prime} \in W_0$, and $x_1^{\prime} 
\in W_1$, there is a unique geodesic in the neighborhood $W$ from 
$x_0^{\prime}$ to $x_1^{\prime}$.

\end{enumerate}
\end{TitleThm}

\par
Then, as an example, it is straightforward to verify that for any 
$z \in \gL^{n+1}$, the vector $z$ is orthogonal to $\gL^{n+1}$ at the point 
$z$.  Suppose $P$ is a plane in $\R^{n+2, 1}$ containing the origin.  Let 
$\gg(t)$ be a constant Lorentzian speed parametrization of the curve obtained 
by intersecting $P$ with $\gL^{n+1}$.  Then, by a standard argument similar 
to that for the case of a Euclidean sphere, $\gg(t)$ is a geodesic. 
All geodesics of $\gL^{n+1}$ are obtained in this way.  It follows that the 
submanifolds of $\gL^{n+1}$ obtained by intersecting $\gL^{n+1}$ with a 
linear subspace is a totally geodesic submanifold of $\gL^{n+1}$.  
\par 
\section{\bf Definition of the Lorentzian Map}
\label{Sec:Lorentz.map} 
\par
	We begin by giving a geometric definition of a Lorentzian map from a 
smooth hypersurface $M \in \R^n$, as a natural map from $M$ to 
$\gL^{n+1}$; and then giving that geometric definition an algebraic form.  
\par 
\subsection{Geometric Definition of the Lorentzian Map} 
\par
First, we let $S^{n}$ denote the unit sphere in 
$\R^{n+1}$ centered at the origin, and we let $\be_{n+1} = (0, \dots, 0, 1)  
\in \R^{n+1}$.  Then, stereographic projection defines a map $p: 
S^{n}\backslash \{\be_{n+1}\} \to \R^n$ sending $y$ to the point where 
the line from $\be_{n+1}$ to $y$ intersects $\R^n$.  Given a hyperplane 
$\Pi$ in $\R^n$, it together with $\be_{n+1}$ spans a hyperplane 
$\Pi^{\prime}$ in $R^{n+1}$.  We can identify $R^{n+1}$ with $R^{n+1} 
\times \{ \be_{n+2}\} \subset R^{n+2, 1}$ by translation in the direction  
$\be_{n+2}$.  The intersection of this plane with $S^{n+1}$ is an 
$n$-sphere.  Then, via this identification of $R^{n+1}$ with the 
hyperplane in $\R^{n+2, 1}$ defined by $x_{n+2} = 1$, we form the 
hyperplane $\Pi^{\prime\prime}$ in $\R^{n+2, 1}$ spanned by 
$\Pi^{\prime}$ together with $0$.  This hyperplane is time-like because 
$\Pi^{\prime\prime}$ intersects $R^{n+1} \times \{ \be_{n+2}\}$ in a 
hyperplane $\Pi^{\prime}$ which intersects the unit sphere in $R^{n+1} 
\times \{ \be_{n+2}\}$ in a sphere, hence it intersects the interior disk.  
Then, the duality defined by the Lorentzian inner product associates to the 
hyperplane $\Pi^{\prime\prime}$ the Lorentzian orthogonal line $\ell$ 
through the origin.  As the hyperplane is time-like, $\ell$ has non-empty 
intersection with $\gL^{n+1}$ in a pair of points $z$ and $-z$.  
\par
In order to obtain a single valued map, there are two possibilities:  Either 
we consider the induce map to $\tilde \gL^{n+1} =  \gL^{n+1}/\sim$, where 
$\sim$ identifies each pair of points $z$ and $-z$  of $\gL^{n+1}$; or we 
need on $\Pi$ a unit vector field $\bn$ orienting $\Pi$.  Given 
the normal vector $\bn$, it defines a distinguished side of $\Pi$.  Then we 
obtain a distinguished side for $\Pi^{\prime}$ and then 
$\Pi^{\prime\prime}$, which singles out one of the two points in 
$\gL^{n+1}$ on the distinguished side.  We shall refer to this second case 
as the oriented case.  We shall use both versions of the maps. \par 
The geometric definition is then as follows. 
\begin{Definition}
\label{Def3.1a}
Given a smooth hypersurface $M \in \R^n$,  with a smooth normal vector 
field $\bn$ on $M$, the {\em (oriented) Lorentz map} is the natural map 
$\cL : M \to \gL^{n+1}$ defined by $\cL (x) = z$, where to $\Pi = T_xM$ is 
associated the plane $\Pi^{\prime\prime}$, Lorentzian orthogonal line 
$\ell$, and the distinguished intersection $z$ with $\gL^{n+1}$.  \par
In the general case where we do not have an orientation for $M$, we define  
$\tilde \cL : M \to \tilde \gL^{n+1}$ by $\tilde \cL (x)$ is the equivalence 
class of $\pm z$ in $\tilde \gL^{n+1}$.  
\end{Definition}
In fact, from the algebraic form of this map to follow, we shall see that it 
actually maps into an $n$ dimensional submanifold $\cT^n$ of $\gL^{n+1}$.
	We give a specific algebraic form for this map. 
\par
\subsection{Algebraic (Coordinate) Definition of the Lorentzian Map} 
\par
We can give a coordinate definitions for the maps.  If $T_xM$ is defined by 
$\bn\cdot \bx  = c$, where $\bx = (x_1, \dots , x_n)$.  Then, 
$\Pi^{\prime}$ contains $T_xM$ and $\be_{n+1}$ and so is defined by 
$\bn\cdot \bx + c x_{n+1}  = c$.  Then, $\Pi^{\prime\prime}$ contains 
$\Pi^{\prime} \times \{ \be_{n+2}\}$ and the origin so it is defined by 
$\bn\cdot \bx + cx_{n+1} - c x_{n+2} = 0$.  Thus, the Lorentzian orthogonal 
line $\ell$ is spanned by $(\bn , c, c)$, which we write in abbreviated 
form as $(\bn , c\bge)$ with $\bge = (1, 1)$.  Hence, the map $\cL : M \to 
\gL^{n+1}$ sends $x$ to $(\bn , c\bge)$, and the general case sends it to 
the equivalence class in $\tilde \cT^n$ determined by $(\bn , c\bge)$.  We 
shall be concerned with a subspace of $\gL^{n+1}$ where this duality 
corresponds to hypersurfaces of $\R^n$.  The general correspondence is 
used in \cite{OH} to parametrize $(n-1)$-dimensional spheres in $\R^n$.  
\par
We need on $M$ a smooth normal unit vector field $\bn$ orienting $M$.  
Given the normal vector field $\bn$, it defines a distinguished side of 
$T_xM$.
\par
In fact, the image lies in the submanifold $\cT^n$ of $\gL^{n+1}$ defined 
by 
$$  \cT^n \quad = \quad \{ (\bn, c\bge) : \bn \in S^{n-1} , c \in \R\}   $$ 
which we can view as a submanifold $\cT^n \subset \gL^{n+1}$; or in the 
general case it lies in $\tilde \cT^n$.  
\begin{Definition}
\label{Def3.1}
Given a smooth hypersurface $M \in \R^n$,  with a smooth normal vector 
field $\bn$ on $M$, the {\em (oriented) Lorentz map} is the natural map 
$\cL : M \to \cT^n$ defined by $\cL (x) = (\bn , c\bge)$, where $T_xM$ is 
defined by $\bn \cdot\bx = c$.  In the general case, we choose a local 
normal vector field  and then $\tilde \cL (x)$ is the equivalence class of 
$(\bn , c\bge)$ in $\tilde \cT^n$.  
\end{Definition}
\par
In the following we shall generally concentrate on the oriented case and 
the map $\cL$, with the general case just involving considering the map to 
equivalence classes.  \par 
Using $\cL$ or $\tilde \cL$, we are led to considering the geodesic flow in 
$\gL^{n+1}$, and obtain the induced geodesic flow on $\tilde \gL^{n+1}$.   
Once we have determined the geodesic flow between points in $\cT^n$, 
there 
are two questions concerning $\cL$ to lift the flow back to hypersurfaces 
in $\R^n$.  One is when $\cL$ is nonsingular, and at singular points what 
can we say about 
the local properties of $\cL$ when $M$ is generic.  The second question is 
how we may construct the inverse of $\cL$ when it is a local embedding 
(or immersion).  

\section{\bf Lorentzian Geodesic Flow on $\gL^{n+1}$}
\label{Sec:geo.deS} 
\par
We give the general formula for the geodesic flow between points  $z_0 = 
(\bn_0, d_0 \bge)$ and $z_1 = (\bn_1, d_1 \bge)$ in $\cT^n$.  \par
\subsection*{Several Auxiliary Functions} \hfill 
\par 
To do so we introduce several auxiliary functions.  
We first define the function $\gl (x, \theta)$ by

\begin{equation}
\label{Eqn3.1} 
 \gl (x, \theta) \quad   =   \quad
\Biggl\{ \begin{matrix} \frac{\sin(x\theta)}{\sin(\theta)}  \quad  & \theta 
\neq 0 \\ \,\, x \quad & \theta = 0
\end{matrix} \Biggr. 
\end{equation}
Then, $\sin(z)$ is a holomorphic function of $z$, and the quotient 
$\frac{\sin(x\theta)}{\sin(\theta)}$ has removable singularities along 
$\theta = 0$ with value $x$.  Hence, $\gl (z, \theta)$ is a holomorphic 
function of $(z, \theta)$ on $\C \times ((-\pi , \pi) \times \iti\,\R)$, and 
so analytic on $\R \times (-\pi , \pi)$.  Also, directly computing the 
derivative we obtain 
\begin{equation}
\label{Eqn3.2} 
 \pd{\gl((x, \theta)}{x}   \quad   =   \quad
\Biggl\{ \begin{matrix}  \cos( x\theta) \cdot \frac{\theta}{\sin \theta} 
 \quad  & \theta \neq 0 \\ \,\, 1 \quad & \theta = 0
\end{matrix} \Biggr. 
\end{equation}
\par
\begin{SimpRemark}
In fact, we can recognize $\gl (n, \theta)$ for integer values $n$ as the 
characters for the irreducible representations of $SU(2)$  restricted to 
the 
maximal torus.
\end{SimpRemark}
We also introduce a second function for later use in \S \ref{Sec:spec.cas}.  
For $- \frac{\pi}{2} < \theta < \frac{\pi}{2}$, we define
$$\mu (x, \theta) \, \, = \,\, \frac{\cos(x\theta)}{\cos(\theta)}.  $$ 
Then, there is the following relation 
\begin{equation}
\label{Eqn3.2a}
\gl(x, \theta) + \gl(1-x, \theta)  \,\, = \,\, \mu (1 -2x, \frac{\theta}{2})   
\end{equation}
This follows by using the basic trigonometric formulas  $\sin (x) + \sin 
(y) =  
2 \cos (\frac{1}{2} ( x - y)) \sin(\frac{1}{2} ( x + y)) $ and $\sin{\theta} = 
2 
\sin (\frac{1}{2}\theta)  \cos (\frac{1}{2}\theta)$.  There are additional 
relations between these two functions that follow from other basic 
trigonometric identies.
\,\, 
\par
\subsection*{Geodesic Curves in $\gL^{n+1}$ joining points in $\cT^n$} 
\hfill 
\par
We may express the geodesic curve between $z_0 = (n_0, c_0 \ge)$ and 
$z_1 = (n_1, c_1 \ge)$ in 
$\gL^{n+1}$ using $\gl (x, \theta)$ provided $n_1 \neq  -n_0$.  
We let $- \frac{\pi}{2} < \theta < \frac{\pi}{2}$ be defined by $\cos \theta 
= \bn_0\cdot \bn_1$.  

\begin{Proposition}
\label{prop4.1}
Provided $n_1 \neq  -n_0$, the geodesic curve $\gg(t)$ in $\gL^{n+1}$ 
between points $\gg(0) = z_0$ and $\gg(1) = z_1$ in $\cT^n$ for the 
Lorentzian metric on $\gL^{n+1}$ is given by
\begin{equation}
\label{Eqn3.4}
  \gg(t) \quad = \quad \gl(t, \theta)\, z_1 \, + \,  \gl(1-t, \theta)\, z_0   
\qquad \mbox{for }   0 \leq t \leq 1
\end{equation}
Furthermore, this curve lies in $\cT^n$ for $0 \leq t \leq 1$.  Hence, 
$\cT^n$ is a geodesic submanifold of $\gL^{n+1}$. 
\end{Proposition}
\par
We can expand the expression for $\gg(t)$ and obtain the family of 
hyperplanes $\Pi_t$ in $\R^n$.   Expanding (\ref{Eqn3.4}) we obtain
\begin{align}
\label{Eqn3.5} 
 n_t \quad &= \quad  \gl(t, \theta)\, \bn_1 \, + \,  \gl(1-t, \theta)\, \bn_0  
\quad and \notag  \\
 c_t  \quad &= \quad  \gl(t, \theta)\, c_1 \, + \,  \gl(1-t, \theta)\, c_0
\end{align}
Then the family $\Pi_t$ is given by 
\begin{equation}
\label{Eqn3.6}
 \Pi_t \quad = \quad \{ \bx = (x_1, \dots , x_n) \in \R^n : \bx \cdot \bn_t 
= c_t\}  
\end{equation}
\par
We can also  compute the initial velocity for the geodesic in 
(\ref{Eqn3.4}).  
\begin{Corollary}
\label{Cor3.7}
The initial velocity of the geodesic (\ref{Eqn3.4}) with $\theta \neq 0$ is 
given by 
\begin{equation}
\label{Eqn3.7}
 \gg^{\prime}(0) \quad = \quad  \frac{\theta}{\sin \theta}\cdot 
(\proj_{\bn_0} (\bn_1), (c_1 - \cos \theta\, c_0) \bge )  
\end{equation} 
where $\proj_{\bn_0}$ denotes projection along $\bn_0$ onto the line  
spanned by $\bw$.  If $\theta = 0$, then $\bn_0 = \bn_1$ and the velocity 
is 
$(0, (c_1 - c_0)\bge)$ (with Lorentzian speed $0$). 
\end{Corollary}
\par
\begin{SimpRemark}
Note that  
$$\| (\proj_{\bn_0} (\bn_1), (c_1 - \cos \theta \, c_0) \bge ) \|_L \,\, = 
\,\,  \| \proj_{\bn_0} (\bn_1)\|$$
which equals $\sin \theta$.  We conclude that the Lorentzian magnitude of 
$\gg^{\prime}(0)$ is $\theta$.  Since geodesics have constant speed, the 
geodesic will travel a distance $| \theta |$.  Hence, $| \theta |$ is the 
Lorentzian distance between $z_0$ and $z_1$.
\end{SimpRemark}
\begin{proof}[Proof of Proposition \ref{prop4.1}]
Let $P$ be the plane in $\R^{n+1, 1}$ which contains $0$, $z_0$ and $z_1$.  
The geodesic curve between $z_0$ and $z_1$ is obtained as a constant 
Lorentzian speed parametrization of the curve obtained by intersecting 
$P$ 
with $\gL^{n+1}$.  We choose a unit vector $\bw \in \Pi$ such that 
$\bn_1$ is in 
the plane through the origin spanned by  $\bn_0$ and $\bw$.  Let $0 \leq 
\theta < \pi$  be the angle between $\bn_0$ and $\bn_1$ so 
$\cos \theta = \bn_0 \cdot \bn_1$.  Then, 
$\bn_1 - (\bn_1\cdot \bn_0)\, \bn_0$ is the projection of 
$\bn_1$ along $\bn_0$ onto the line spanned by $\bw$ whose direction is 
chosen so that  $\bn_1 - \cos \theta\, \bn_0 = \sin \theta\, \bw$.  \par
Then, a tangent vector to $\gL^{n+1} \cap P$  at the point $z_0$ is given by 
\begin{equation}
\label{Eqn3.11}
  (\bn_1 - \cos \theta\, \bn_0, (c_1 - \cos \theta\, c_0) \bge)  = (\sin 
\theta\, \bw, (c_1 - \cos \theta\, c_0) \bge) 
\end{equation}
Then, we seek a Lorentzian geodesic $\gg(t)$ in the plane $P$ beginning at 
$(\bn_0, c_0 \bge)$ with initial velocity in the direction $(\sin \theta\, 
\bw, 
(c_1 - \cos \theta\, c_0) \bge)$.  Consider the curve 
\begin{equation}
\label{Eqn3.12}
 \gg(t) \,\, = \,\, (\cos(t \theta) \bn_0 \, + \, \sin (t \theta) \bw, (\cos(t 
\theta) c_0 \, + \, \frac{\sin (t \theta)}{\sin (\theta)}(c_1 - \cos \theta\, 
c_0)) \bge) 
\end{equation}
First, note that $\gg(0) = z_0$, and $\gg(1) = z_1$.  Also, this curve lies
 in the plane spanned by $z_0$ and (\ref{Eqn3.11}).  Also, 
$$ \| \gg(t)\|_L \,\, = \,\, \| \cos(t \theta) \bn_0 \, + \, \sin (t \theta) 
\bw\| \,\, = \,\, 1  $$
as $\bn_0$ and $\bw$ are orthogonal unit vectors.  Hence, $\gg(t)$ is a 
curve parametrizing $\gL^{n+1} \cap P$.  It remains to show that 
$\gg^{\prime \prime}$ is Lorentzian orthogonal to $\gL^{n+1}$ to establish 
that it is a Lorentzian geodesic from $z_0$ to $z_1$.
A computation shows 
$$ \gg^{\prime\prime}(t) \,\, = \,\, - \theta^2 ( \cos(t \theta) \bn_0 \, + 
\, \sin (t \theta) \bw, \frac{\sin (t \theta)}{\sin (\theta)} (c_1 - \cos 
\theta\, c_0) \bge) $$
which is $ - \theta^2\gg(t)$, and hence Lorentzian orthogonal to 
$\gL^{n+1}$.  
\par
Because of the fraction $\frac{\sin (t \theta)}{\sin (\theta)}$, we have to 
note that when $\theta = 0$, then $\bn_0 = \bn_1$ and $\gg(t)$ takes the 
simplified form
$$  \gg(t) \,\, = \,\, (\bn_0, c_0 \, + \, t (c_1 - c_0)) \bge) $$
which is still a Lorentzian geodesic between $z_0$ to $z_1$.  \par
Lastly, we must show that this agrees with (\ref{Eqn3.4}).  First, consider 
the case where $\theta \neq 0$.
$$\bw = \frac{1}{\sin \theta}\, (\bn_1 - \cos \theta\, \bn_0)$$
Substituting this into the first term of the RHS of (\ref{Eqn3.12}), we 
obtain
$$  \frac{1}{\sin \theta}(\sin \theta\, \cos(t \theta) - \cos \theta\, \sin 
(t 
\theta))\, \bn_0 \, + \, \frac{sin (t \theta)}{\sin \theta}\, \bn_1 $$
which by the formula for the sine of the difference of two angles  equals 
$$  \frac{\sin((1-t)\theta)}{\sin \theta}  \bn_0 \, +  \frac{sin (t 
\theta)}{\sin \theta} \bn_1 $$
Analogously, we can compute the second term in the RHS of 
(\ref{Eqn3.12}),
to be
$$   \frac{\sin((1-t)\theta)}{\sin \theta}  c_0 \, +  \frac{sin (t 
\theta)}{\sin \theta} c_1   $$
This gives (\ref{Eqn3.4}) when $\theta \neq 0$.  When $\theta = 0$, $\bn_0 
= \bn_1$ and the derivation of (\ref{Eqn3.4}) from (\ref{Eqn3.12}) for 
$\theta = 0$ is easier.
\end{proof}
\par
\begin{Remark}
\label{Rem4.3}
We have alread seen that the geodesic flow between the planes 
$\bn_0\cdot x = c_0$ and $\bn_1\cdot x = c_1$ induced from the geodesic 
flow in $\R P^{n \vee}$ corresponds to the geodesic flow between $(\bn_0, 
c_0)$ and $(\bn_1, c_1)$, which is given by the unit speed curve in the 
intersection of the plane $P$, containing these points and the origin, with 
the unit sphere $S^n$.  If we replaced (\ref{Eqn3.4}) by linear interpolation
\begin{equation}
\label{Eqn3.13}
  \gg(t) \quad = \quad t\, (\bn_1, c_1) \, + \,  (1-t)\, (\bn_0, c_0)   
\qquad \mbox{for }   0 \leq t \leq 1
\end{equation}
then the curve lies in the plane $P$ and its projection onto the unit sphere 
does parametrize the geodesic, but it is not unit speed, and as we 
remarked earlier it is not invariant under translation and hence not under 
rigid transformations. 
\end{Remark}
\section{\bf Invariance of Lorentzian Geodesic Flow and Special Cases}
\label{Sec:spec.cas} 
\par
We investigate the invariance properties of Lorentzian geodesic flows and 
the properties of these flows in special cases.
\subsection*{Invariance of Lorentzian Geodesic Flow} 
\par
We first claim the Geodesic flow given in Proposition \ref{prop4.1} is 
invariant under the extended Poincare group generated by rigid 
transformations and scalar multiplications.  By this we mean 
the following.  
If $\gg(t) = (\bn_t, c_t)$ is the Lorentzian geodesic flow between 
hyperplanes $P_0$ and $P_1$ defined by $\bn_0\cdot x = c_0$, 
respectively $\bn_1\cdot x = c_1$, then $\tilde \psi (\gg(t))$ is the 
Lorentzian geodesic flow between 
hyperplanes $\psi(P_0)$ and $\psi(P_1)$.
\begin{Proposition}
\label{prop5.1}
The Lorentzian geodesic flow is invariant under the extended Poincare 
group.
\end{Proposition}
\begin{proof}
Suppose $z_i = (\bn_i, c_i) \in \cT^n$, i = 1, 2, and let $\Pi_i$ be the 
hyperplane determined by $z_i$.  Let $\psi$ be an alement of the extended 
Poincare group.  It is a composition of scalar multiplication by $b$ 
followed by a rigid transformation so $\psi(\bx) = b\, A(\bx) + \bp$, with 
$A$ an orthogonal transformation.  Then, 
$\Pi_i^{\prime} = \psi (\Pi_i)$ is defined by 
\begin{equation}
\label{Eqn5.2a}
\tilde \psi (z_i) = \tilde \psi (\bn_i, c_i) = (A(\bn_i), bc_i + \bn_i\cdot 
\bp).
\end{equation}
\par
If $\cos(\theta) = \bn_0\cdot \bn_1$, then by (\ref{Eqn3.4}) the 
Lorentzian geodesic flow is given by $\gamma(t)$ defined by
\begin{equation}
\label{Eqn5.3}
 (\bn_t, c_t) \,\, = \,\, (\gl(t, \theta)\bn_1 \, + \,  \gl(1-t, \theta)\, 
\bn_0,\,\, \gl(t, \theta)\, c_1 \, + \,  \gl(1-t, \theta)\, c_0 ) 
\end{equation}
defining the family of hyperplanes $\Pi_t$.  Then, by \eqref{Eqn5.2a} 
$\Pi_t^{\prime} = \psi(\Pi_t)$ is defined by $\bn^{\prime}_t\cdot x = 
c^{\prime}_t$, where $\tilde \psi(\gamma(t))$ is defined by
\begin{align}
\label{Eqn5.4}
  (\bn^{\prime}_t, c^{\prime}_t) \,\, &= \,\, \left( A(\gl(t, \theta)\bn_1 \, 
+ \,  \gl(1-t, \theta)\, 
\bn_0), \, b(\gl(t, \theta))c_1 + \gl(1-t, \theta))c_0) \, + \right. \\
& \qquad \qquad  \, \left. A(\gl(t, \theta)\bn_1 \, + \,  \gl(1-t, 
\theta)\cdot \bp)\right) \notag  \\
&= \,\,\gl(t, \theta)\left(A(\bn_1), \, b c_1 + A(\bn_1)\cdot \bp\right) \, 
+ \,
\gl(1-t, \theta)\left( A(\bn_0),  \, b c_0 + A(\bn_0)\cdot \bp\right) 
\notag \\ 
&= \,\, \gl(1-t, \theta)\tilde \psi(z_1) + \gl(1-t, \theta)\tilde \psi(z_0) 
\notag  
\end{align}
which is the geodesic flow between $\Pi_0^{\prime}$ defined by $\tilde 
\psi(z_0)$ and $\Pi_1^{\prime}$ defined by $\tilde \psi(z_1)$.  
\par
\end{proof}
\begin{Remark}
\label{Rem5.5}
An alternate way to understand Proposition \ref{prop4.1} is to observe 
that the extended Poincare group acts on $\R^n \times \R \ge$ sending 
$(\bv, c\ge) \mapsto (A(\bv), (b c + \bv\cdot \bw)\ge)$.  This action 
preserves the Lorentzian inner product on this subspace and preserves 
$\cT^n$.  Hence, it maps geodesics in $\cT^n$ to geodesics in $\cT^n$.  
\end{Remark}
\subsection*{Special Cases of Lorentzian Geodesic Flow}
\par
We next determine the form of the Lorentzian geodesic flow in several 
special cases. \par
\begin{Example}[Hypersurfaces Obtained by a Translation and Homothety] 
\par
Suppose that we obtain $\Pi_1$ from $\Pi_0$ by translation by a vector 
$\bp$ and multiplication by a scalar $b$.  The correspondence associates 
to $\bx \in \Pi_0$, $bx^{\prime} = \bx + \bp \in 
\Pi_1$.  Then, the geodesic flow is given by the following.
\begin{Corollary}
\label{Cor5.6}
Suppose $\Pi_0$ is the hyperplane defined by $\bn_0\cdot x = c_0$, with 
$\bn_0$ a unit vector, and $\Pi_1$ is obtained from $\Pi_0$ by 
multiplication by the scalar $b \neq 0$ and then translation by $\bp$.  Then 
the Lorentzian geodesic 
flow $\Pi_t$ is given by the family of parallel hyperplanes defined by 
$\bn_0\cdot x = c_t$ where $c_t = (1+ (b - 1)t)c_0 + t \bn_0\cdot \bp$.  
\end{Corollary}
\par
\begin{proof}
\par
If  $\Pi_0$ is defined by $\bn_0\cdot x = c_0$, with $\bn_0$ a unit 
vector, then, $\Pi_1$ is defined by $\bn_1\cdot x = c_1$ where $\bn_1 = 
\bn_0$ and $c_1 = b c_0 + \bn_0 \cdot \bp$.  Thus, $\Pi_1$ is parallel to 
$\Pi_0$.  \par 
Thus, as $\bn_1 = \bn_0$, $\theta = 0$ and $\gl(t, 0) = t$, so the geodesic 
flow $\Pi_t$ is given by 
\begin{align}
\label{Eqn5.8}
  t (\bn_0, c_1 \bge)  + (1-t) (\bn_0, c_0 \bge)  \,\, &= \,\,  (\bn_0, ((1-
t)c_0 + t (b c_0 + \bn_0 \cdot \bp)) \bge ) \notag\\  
\,\, &= \,\, (\bn_0, ((t b + 1-t)c_0 + t (\bn_0 \cdot \bp)) \bge) 
\end{align}
so that $\Pi_t$ is defined by $\bn_0\cdot \bx = c_t$ where $c_t = (1 + (b 
- 1)t)c_0 + t (\bn_0 \cdot \bp)$.  

This defines a family of hyperplanes parallel to $\Pi_0$ where 
derivative of the translation map is the identity; hence, under translation 
$\bn_0$ is mapped to itself translated to $\bx^{\prime} = \bx + \bp$.  
Thus, under the correspondence, $\bn_1 = \bn_0$.  Also, If $\bn_0 \cdot 
\bx = c_0$ is the equation of the tangent plane for $M_0$ at a point $\bx$, 
then the tangent plane of $M_1$ at the point $\bx^{\prime}$ is 
$$ \bn_1 \cdot  \bx^{\prime} \, = \, \bn_0 \cdot  (\bx + \bp) \, = \, c_0 + 
\bn_0 \cdot \bp  $$
Hence, $c_1 = c_0 + \bn_0 \cdot \bp$.  \par
As $\bn_0 = \bn_1$, $\theta = 0$.  Thus the geodesic flow on $\cT^n$ is 
given 
by 
$$  t (\bn_0, c_1 \bge)  + (1-t) (\bn_0, c_0 \bge)  \,\, = \,\,  (\bn_0, c_0 
\bge) + (0, (t \bn_0 \cdot \bp) \bge )  \,\, = \,\, (\bn_0, (\bn_0 \cdot (\bx 
+  t \bp)) \bge)  $$
Thus, at time $t$ the tangent space is translated by $t \bp$.  Thus, the 
envelope of these translated hyperplanes is the translation of $M_0$ by $t 
\bp$.  
\end{proof}
\begin{Remark}
If a hypersurface $M_1$ is obtained from the hypersurface $M_0$ by a 
translation combined with a homothety $\bx^{\prime} = \psi(\bx) = b \bx + 
\bp$, then for each $x \in M_0$ with image $x^{\prime} \in M_1$ the 
Lorentzian geodesic flow will send the tangent plane $T_xM_0$ to the 
tangent plane $T_{x^{\prime}}M_1$  by the family of parallel hyperplanes 
given by Corollary \ref{Cor5.6}.  Thus, for each $0 \leq t \leq 1$, the 
hyperplane under the geodesic flow will be the tangent plane 
$T_{\bx_t}M_t$, where for $\psi_t(\bx) = (1 + (b - 1) t) \bx + t \bp$, 
$M_t = \psi_t(M_0)$ and $\bx_t = \psi_t(\bx) \in M_t$.  Thus, the Lorentzian 
geodesic flow will send $M_0$ to the family of hypersurfaces 
$M_t = \psi_t(M_0)$. 
\end{Remark}
\end{Example}

\par
\begin{Example}[Hyperplanes Obtained by a Pseudo-Rotation]
\par
\begin{figure}[ht]
\centerline{\includegraphics[width=10cm]{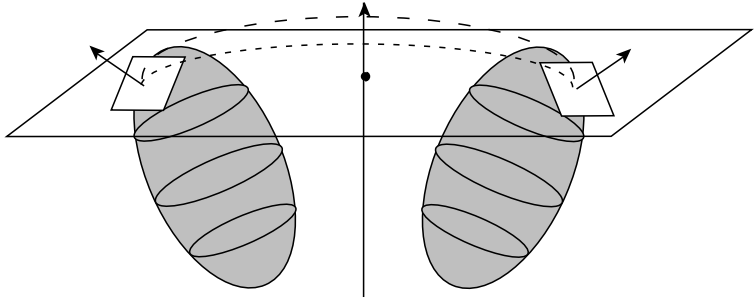}}
\caption{\label{fig.5}  Lorentzian Geodesic Flow between a hyperplane 
$\Pi_0$ and a rotated copy $\Pi_1$, where the rotation is about a 
subspace not containing $W = \Pi_0 \cap \Pi_1$, is given by a 
\lq\lq pseudo--rotation\rq\rq.  The path of the 
rotation is indicated by the dotted curve, while that for the pseudo 
rotation is given by the  curve, which lifts out of the plane of rotation 
before returning to it (although it does remain in a plane parallel to 
$W^{\perp}$).}
\end{figure} 
\par
Second, suppose that $\Pi_0$ and $\Pi_1$ are nonparallel affine hyperplanes.  
Then, $W = \Pi_0 \cap \Pi_1$ is a codimension $2$ affine subspace.  The 
unit normal vectors $\bn_0$ and $\bn_1$ lie in the orthogonal plane 
$W^{\perp}$ with $\bn_0 \cdot \bn_1 = \cos(\theta)$ with $-\pi/2 < 
\theta < \pi/2$.  Since the Lorentzian geodesic flow commutes with 
translation, we may translate the planes and assume that $W$ contains 
the origin.  Then, both $c_0$ and $c_1$ equal $0$.  Thus, by Proposition 
\ref{prop4.1}, the Lorentzian geodesic flow from $\Pi_0$ to $\Pi_1$ is 
given by $(\bn_t, c_t\bge)$ for $0 \leq t \leq 1$, where
\begin{equation}
\label{Eqn5.9}
  (\bn_t, c_t\bge)  \,\, = \,\,  (\gl(t, \theta)\, (\bn_1, c_1\bge) \, + \,  \gl(1-t, 
\theta)\, (\bn_0, c_0\bge)   
\end{equation}
Thus, $\bn_t = \gl(t, \theta) \bn_1 + \gl(1-t, \theta) \bn_0$, while $c_t \equiv 
0$.  Hence, the hyperplane $\Pi_t$ is defined by $\bn_t \cdot \bx = 0$ so it 
contains $W$.  However, its intersection with the plane $W^{\perp}$ is the 
line orthogonal to $\bn_t$, which by the above expression for $\bn_t$, does 
not give a standard constant speed rotation in the plane.  We refer to this 
as a {\em pseudo-rotation}. 
\par
Instead consider a rotation $A$ of hyperplanes $\Pi_0$ to $\Pi_1$ about 
an axis not containing $W$.  We consider the form of the pseudo-rotation.  
As an example, consider the case of a rotation $A$ about the origin in a 
plane (which pointwise fixes an orthogonal subspace.  Choosing 
coordinates, we may assume that the rotation $A$ is in the 
$(x_1, x_2)$--plane and rotates by an angle $\gw$.  We suppose $\Pi_0$, 
defined by $\bn_0 \cdot \bx = c_0$.  As $A (\bn_0) \cdot A(\bx) = \bn_0 
\cdot \bx = c_0$, if we let $\bx^{\prime} = A(\bx)$, then the equation of 
the hyperplane $\Pi_1$  is defined by $A (\bn_0) \cdot \bx^{\prime} = 
c_0$.  
Hence, $\bn_1 = A (\bn_0)$ and $c_1 = c_0$.  \par
To express the geodesic flow, we write $\bn_0 = \bv + \bp$ where $\bv$ 
is in the rotation plane and $\bp$ is fixed by $A$.  Hence, $\bn_1 = A(\bv) 
+ \bp$.  Thus, the angle $\theta$ between $\bn_0$ and $\bn_1$ satisfies 
$$ \cos \theta \, = \, \bn_1 \cdot \bn_0 \, = \, A(\bv)\cdot \bv \, + \, 
\bp\cdot \bp  $$
As $\| \bn_0 \| = 1$, we obtain $\bv\cdot \bv \, + \, \bp\cdot \bp = 1$.  
Also, $A(\bv)\cdot \bv = \| \bv \|^2 \cos \gw$.   Hence, 
\begin{equation}
\label{Eqn5.10}
  \cos \theta \, = \,  1 + \| \bv \|^2 (\cos \gw - 1)  
\end{equation}
\par
We recall that by (\ref{Eqn3.2a})

\begin{equation*}
\gl(t, \theta) + \gl(1-t, \theta)  \,\, = \,\, \mu (1 -2t, \frac{\theta}{2})   
\end{equation*}

Using the expressions for $\bn_0$ and $\bn_1$, we find the geodesic flow 
is 
given by
\begin{align}
\label{Eqn5.11}
  &= \,\, \gl(t, \theta)\, (A(\bn_0), c_0 \bge)  +  \gl(1-t, \theta)\, (\bn_0, 
c_0 \bge)  \notag \\ 
 &= \,\,  
((\gl(t, \theta) A(\bv) + \gl(1-t, \theta)\, \bv) \, + \,  \mu (1 -2t, 
\frac{\theta}{2})\, \bp, \,\, \mu (1 -2t, \frac{\theta}{2}) c_0 \bge)  
\end{align}
\par
We note that $\mu (1 -2t, \frac{\theta}{2})$ is a function of $t$ on 
$[0, 1]$ which has value 
$ = 1$ at the end points, and has a maximum $ = \sec (\frac{1}{2}\theta)$ 
at $t = \frac{1}{2}$.  Thus, the geodesic flow $(\bn_t, c_t\bge)$ has the 
contribution in the rotation plane given by 
$\gl(t, \theta) A(\bv) + \gl(1-t, \theta) \bv$ which is not a true rotation 
from $\bv$ to $A(\bv)$.  Also, the other contribution to $\bn_t$ is from 
$\mu (1 -2t, \frac{\theta}{2}) \bp$ which increases and then returns to 
size $\bp$ (see Fig. \ref{fig.5}).  In addition, the distance from the origin 
will vary by $\mu (1 -2t, \frac{\theta}{2}) c_0$.  This is the form of the  
pseudo rotation from $\Pi_0$ to $\Pi_1$.  
This yields the following corollary.
\begin{Corollary}
\label{Cor5.12}
If $\Pi_1$ is obtained from $\Pi_0$ by rotation in a plane (with fixed 
orthogonal complement), then the Lorentzian geodesic flow is the family 
of hypersurfaces obtained by applying to $\Pi_0$ the family of pseudo 
rotations given by (\ref{Eqn5.11}).  
\end{Corollary}

\end{Example}
\par
\section{\bf Families of Lorentzian Parallel Frames on Lorentzian Geodesic 
Flows}
\label{Sec:par.frms}
\par
A Lorentzian geodesic flow from hyperplanes $\Pi_0$ to $\Pi_1$ may be 
viewed as a minimum twisting family of hyperplanes $\Pi_t$ joining 
$\Pi_0$ to $\Pi_1$.  If in addition, we are given orthonormal frames 
$\{ e_{i\, 0}\}$ for $\Pi_0$ and $\{ e_{i\, 1}\}$ for $\Pi_1$, we ask what 
form a minimum twisting family of smoothly varying frames $\{ e_{i\, 
t}\}$ for $\Pi_t$ should take?  We give the form of the family of \lq\lq 
Lorentzian parallel\rq\rq\, orthonormal frames in $\Pi_t$ beginning with 
$\{ e_{i\, 0}\}$, and then use this family to construct a family of frames 
from $\{ e_{i\, 0}\}$ to $\{ e_{i\, 1}\}$ which can be made to satisfy 
various criteria for minimal Lorentzian twisting.  \par

\vspace{1ex}
\par
\subsection*{\bf Criterion for Lorentzian Parallel Vector Fields}  \hfill

\par
Given a smooth curve $\gg(t)$, $0 \leq t \leq 1$ in $\R^n$ and a smoothly 
varying  family of affine hyperplanes $\{\Pi_t\}$ satisfying: 
\begin{itemize}
\item[1)] $\gg(t) \in \Pi_t$ for each $t$;
\item[2)] $\gg(t)$ is tranverse $\Pi_t$ for each $t$.
\end{itemize}
\par
We let $\bn_t$ denote the smooth family of unit normals to the hyperplanes 
$\Pi_t$.  Then there is a corresponding curve in $\gL^{n+1}$ defined by 
$\tilde{\gg}(t) = (\bn_t, c_t \bge)$ where $c_t = <\gg(t), \bn_t>$.  Let 
$\be_t$ denote a smooth section of $\{\Pi_t\}$, by which we mean that if 
we view $\be_t$ as a vector from the point $\gg(t)$ lies in the hyperplane 
$\Pi_t$ for each $t$.  There is then a corresponding vector field 
$\tilde{\be}_t$ on $\tilde{\gg}(t)$defined by $\tilde{\be}_t = (\be_t, \gb(t) \bge)$.  
This vector field is tangent to $\gL^{n+1}$ as the vector $N_t = (\bn_t, c_t \bge)$ 
is Lorentzian normal to $\gL^{n+1}$ at $\tilde{\gg}(t)$ so $< N_t, \tilde{\be}_t>_L = <\bn_t, \be_t> = 0$.  \par
We give a criterion for $\tilde{\be}_t$ to be a Lorentzian parallel vector field along 
$\tilde{\gg}(t)$. 
\par
\begin{Lemma}[Criterion for Lorentzian Parallel Vector Fields]
\label{Lem6.1a}
The smooth vector field $\tilde{\be}_t$ is Lorentzian parallel along $\tilde{\gg}(t)$ if :
\begin{itemize}
\item[i)]  $\pd{e_t}{t} = \varphi(t) \bn_t$ for a smooth function $\varphi(t)$
\item[ii)] $\gb(t) = \int \varphi(t) c_t \,dt$ for each $t$.
\end{itemize}
\end{Lemma}
\par
\begin{proof}
As $N_t$ is Lorentzian normal to $\gL^{n+1}$ at $\tilde{\gg}(t)$, it is sufficient to show that $\pd{\tilde{e}_t}{t} = \ga(t) N_t$ for some function $\ga(t)$.  Then, by i) and ii)
\begin{align}
 \pd{\tilde{e}_t}{t} \,\, &= \,\, (\pd{e_t}{t}, \gb^{\prime}(t)) \,\, = \,\, (\varphi(t) \bn_t, \gb^{\prime}(t) \bge)  \notag  \\  
\,\, &= \,\, \varphi(t) (\bn_t, c_t \bge) \,\, = \,\, \varphi(t) N_t 
\end{align}
Hence, $\tilde{e}_t$ is Lorentzian parallel.
\end{proof}
\par 
Hence, a smooth section $\be_t$ of $\{\Pi_t\}$ extends to a Lorentzian parallel vector field $\tilde{\be}_t$ provided condition i) is satisfied and using condition ii) to define 
$\gb(t)$. \par
\vspace{2ex}
\begin{Example}
\label{Exam6.2a}
Suppose $\Pi_t$ is the normal hyperplane to $\gg(t)$ at the point $\gg(t)$ for each $t$.  Then the condition that $\be_t$ is a section of $\Pi_t$ is that $<\be_t, \gg^{\prime}(t)> = 0$.  Then, by Lemma \ref{Lem6.1a} the condition that $\be_t$ is moreover a parallel vector field is that there is a smooth function $\varphi(t)$ so that $\pd{\be_t}{t} = \varphi(t) \gg^{\prime}(t)$.  These two conditions are the criteria in \cite{WJZY} and other papers quoted there that for the normal family of affine planes in $\R^3$ the vector field $\be_t$ has \lq\lq minimum rotation\rq\rq.   
\end{Example}
\begin{Remark}
In the case when the family of affine hyperplanes $\{\Pi_t\}$ are not normal, then the vectors $\bn_t$ and $\gg^{\prime}(t)$ are not parallel so the condition in Lemma \ref{Lem6.1a} replaces the role  of $\gg^{\prime}(t)$ in both conditions by $\bn_t$.
\end{Remark}

Then, for each such vector field $\zeta(t)$ and smooth function $\gb(t)$, 
we define a smooth tangent vector field to $\gL^{n+1}$ (in fact $\cT^n$) 
along $\gg(t)$ by $\tilde \zeta(t) = (\zeta(t), \gb(t) \bge)$.  We 
observe that at each point $(\bn_t, c_t \bge)$, $\tilde \zeta(t)$ is 
Lorentzian orthogonal to $(\bn_t, c_t \bge)$ and so is tangent to 
$\gL^{n+1}$.  Moreover, because of the form of $\tilde \zeta(t)$, it is also 
tangent to $\cT^n$.  However, there may be no specific choice of $\gb(t)$ 
possible for $\tilde \zeta(t)$ to be a Lorentzian parallel vector field on 
$\cT^n$.  We say that the smooth section $\zeta(t)$ of $\{\Pi_t\}$ is a 
{\em Lorentzian parallel vector field} if there is a smooth function 
$\gb(t)$ so that $\tilde \zeta(t)$ is a Lorentzian parallel vector field on 
$\cT^n$.  For example, in the special case that the section $\zeta(t) \equiv \bv$ is 
constant we may choose $\gb(t) \equiv 0$, and the vector field $\tilde 
\zeta(t) = (\zeta(t), 0\cdot \bge)$ is constant and so is a Lorentzian 
parallel vector field.  
\par
Thus, given a set of such vector fields $\{\zeta_i(t): i = 1, \dots , k\}$  
which are sections of $\{ \Pi_t\}$ together with smooth functions 
$\gb_i(t)$, then we obtain a set of vector fields $\{\tilde \zeta_i(t): i = 1, 
\dots , k\}$ on $\gg(t)$ tangent to $\cT^n$.  Then, the existence of 
Lorentzian parallel families of frames for 
$\{ \Pi_t\}$is given by the following.
\begin{Proposition}
\label{Prop6.1}
Let $\gg(t) = (\bn_t, c_t \bge)$ be a Lorentzian geodesic defining the 
family of hyperplanes $\{ \Pi_t\}$ in $\R^n$.  If $\{ e_{i\, 0}, 1 \leq i \leq 
n-1\}$ is an orthonormal frame for $\Pi_0$, there is a (smoothly varying) 
family of orthonormal frames $\{ e_{i\, t}, , 1 \leq i \leq n-1\}$ for $\{ 
\Pi_t\}$ such that the vector fields $\{ \tilde e_{i\, t}, 1 \leq i \leq n-
1\}$ form a family of Lorentzian parallel vector fields on $\cT^n$ which 
are Lorentzian orthonormal.  
\end{Proposition}
\begin{proof}
\par 
First, if $\Pi_0$ and $\Pi_1$ are parallel then $\Pi_1$ is a translation of 
$\Pi_0$, so by Corollary \ref{Cor5.6} the Lorentzian geodesic flow 
$\Pi_t$ is a family of hyperplanes parallel to $\Pi_0$ so the family of 
frames is the \lq\lq constant\rq\rq\, family obtained by translating $\{ 
e_{i\, 0}\}$ to each hyperplane in the family.  The corresponding family 
$\{ \tilde e_{i\, t}\}$ is also constant, and hence Lorentzian parallel. 
\par
Next we consider the case where $\Pi_0$ and $\Pi_1$ are not parallel.  We 
first construct the required Lorentzian parallel family beginning with a 
specific orthonormal frame for $\Pi_0$.  Then, we explain how to modify 
this for a general orthonormal frame.  \par
We have $\Pi_0$ is defined by $\bn_0\cdot \bx = c_0$, and $\Pi_1$, by 
$\bn_1\cdot \bx = c_1$ with $\bn_0\cdot \bn_1 = \cos(\theta)$ for $-
\pi/2 < \theta < \pi/2$.  Then, as earlier in \S \ref{Sec:spec.cas}, $W = 
\Pi_0 \cap\Pi_1$ is a codimension $2$ affine subspace, and every 
hyperplane $\Pi_t$ in the Lorentzian geodesic flow $\gg(t) = (\bn_t, c_t 
\bge)$ from $\Pi_0$ to $\Pi_1$ contains $W$.  \par  
We let $e_2, \dots , e_{n-1}$ denote an orthonormal frame for $W$.  Then, 
the $e_i$ define constant vector fields $e_i$ along the Lorentzian 
geodesic with each $e_i \in W \subset \Pi_t$. These allow us to define 
$\tilde e_i$, 
which are parallel vector fields on $\gL^{n+1}$ (in fact $\cT^n$ along the 
Lorentzian geodesic $\gg(t)$.  \par 
Hence, to complete them to an orthonormal frame, we need only construct a 
unit vector field $e_{1, t}$ which is a smooth section of $\{\Pi_t\}$ 
orthogonal to $W$ for each $0 \leq t \leq 1$ and show that the resulting 
vector field $\tilde e_{1, t}$ is a Lorentzian parallel vector field on 
$\gg(t)$.  \par  
The subspace of any $\Pi_t$ orthogonal to $W$ is one dimensional, so 
there are two choices for a unit vector spanning it.   For $\Pi_0$ we 
choose $e_{1, 0}$ so that $\bn_0, e_{1, 0}, e_2, \dots , e_{n-1}$ is 
positively oriented for $\R^n$.  Likewise we choose $e_{1, 1}$ for $\Pi_1$ 
so that $\bn_1, e_{1, 1}, e_2, \dots , e_{n-1}$ is positively oriented for 
$\R^n$.  Then, we define 
\begin{equation}
\label{Eqn6.1a} e_{1, t} \,\, = \,\, \gl(t, \theta) e_{1, 1} \, + \, \gl(1-t, 
\theta) e_{1, 0}  
\end{equation}
We first claim that $e_{1, t}$ is a unit vector field such that $e_{1, t} \in 
\tilde \Pi_t$ for all $0 \leq t \leq 1$.  \par
That $e_{1, t}$ is a unit vector field follows from the calculation for 
$\bn_t$ in Proposition \ref{prop4.1}.  Second, we compute 
\begin{align}
\label{Eqn6.1}
e_{1, t}\cdot \bn_t \,\, &= \,\, (\gl(t, \theta) e_{1, 1} \, + \, \gl(1-t, 
\theta) e_{1, 0}) \cdot (\gl(t, \theta) \bn_1 \, + \, \gl(1-t, \theta) \bn_0)  
\notag \\
&= \,\, (\gl(t, \theta) \gl(1-t, \theta)) (e_{1, 1} \cdot \bn_0 + e_{1, 0} 
\cdot \bn_1)
\end{align}
To see the RHS of \eqref{Eqn6.1} is zero, we consider the two positively 
oriented orthonormal bases for $W^{\perp}$: $\bn_0, e_{1, 0}$ and $\bn_1, 
e_{1, 1}$.  If we represent $\bn_1 = a \bn_0 + b e_{1, 0}$, then necessarily 
$e_{1, 1} = -b \bn_0 + a e_{1, 0}$.  Then, 
$$ e_{1, 1} \cdot \bn_0 + e_{1, 0} \cdot \bn_1 \,\, = \,\, -b + b  \,\, = \,\, 
0 $$
\par
We also note that $e_{1, t}$ is orthogonal to $\tilde W$ for all $t$.  Thus, 
the resulting tangential vector fields $\tilde e_{1, t}, \tilde e_2, \dots , 
\tilde e_{n-1}$ are mutually Lorentz orthogonal and are Lorentzian unit 
vector fields.  The vector fields $\tilde e_i$, $i = 2, \dots , n-1$ are 
constant and hence Lorentzian parallel.  It remains to show that $\tilde 
e_{1, t}$ is Lorentzian parallel.  We claim that if $\gb(t) = 
\frac{1}{\theta} c^{\prime}_t$, then $\tilde e_{1, t} = (e_{1, t}, \gb(t) 
\cdot \ge)$ is a Lorentzian parallel tangent vector field along $\gg(t)$.  
As $\gg(t)$ is a Lorentzian geodesic, $\gg^{\prime}(t)$ is Lorentzian 
parallel along $\gg(t)$.  We will show that with the given $\gb(t)$, $\tilde 
e_{1, t} = \frac{1}{\theta} \gg^{\prime}(t)$.  
\par
From the proof of Proposition \ref{prop4.1}, by \eqref{Eqn3.12}, $\gg(t)$ 
can be written 
\begin{equation}
\label{Eqn6.2a}
 \gg(t) \,\, = \,\, (\cos(t \theta) \bn_0 \, + \, \sin (t \theta) \bw, (\cos(t 
\theta) c_0 \, + \, \frac{\sin (t \theta)}{\sin (\theta)}(c_1 - \cos \theta\, 
c_0)) \bge) 
\end{equation}
Hence, 
\begin{equation}
\label{Eqn6.3a}
 \gg^{\prime}(t) \,\, = \,\, \theta (-\sin(t \theta) \bn_0 \, + \, \cos (t 
\theta) \bw, (-\sin(t \theta) c_0 \, + \, 
\frac{\cos (t \theta)}{\sin (\theta)}(c_1 - \cos \theta\, 
c_0)) \bge) 
\end{equation}
Both $\{\bn_0, \bn^{\perp}_0\}$ and $\{\bn_1, \bn^{\perp}_1\}$ have 
positive orientation in $W^{\perp}$; hence $e_{1, 0} = \bn^{\perp}_0$ and 
$e_{1, 1} = \bn^{\perp}_1$   If we represent $\bn_1 = a \bn_0 + b 
\bn^{\perp}_0$, then $\bn^{\perp}_1 = -b \bn_0 + a \bn^{\perp}_0$ and in 
\eqref{Eqn6.3a} the unit vector $\bw = \bn^{\perp}_0$ if $b > 0$ and $-
\bn^{\perp}_0$ if $b < 0$.  \par  
Second we represent $e_{1, t}$ in the same form as \eqref{Eqn6.2a}.  To do 
so we compute the unit vector in the same direction as the projection of 
$e_{1, 1} (= \bn^{\perp}_1)$ along $e_{1, 0} (= \bn^{\perp}_0)$.  
Then, by the above, $e_{1, 1} = -b \bn_0 + a e_{1, 0}$.  Thus, the 
corresponding $\bw$ for this case is either $\bw_1 =  - \bn^{\perp}_0$ if 
$b > 0$ or $\bn^{\perp}_0$ if $b < 0$.  Thus, by the calculation in the proof 
of Proposition \ref{prop4.1}, in either case we obtain
\begin{align}
\label{Eqn6.4a}
 e_{1, t} \,\, &= \,\, \cos (t \theta) \bn^{\perp}_0  \, + \, \sin(t \theta) 
e_{1, 0}
\notag  \\
\,\, &= \,\, -\sin(t \theta) \bn_0 \, + \, \cos (t \theta) \bn^{\perp}_0
\end{align}
This equals the first component of \eqref{Eqn6.3a}, which implies the 
equality $e_{1, t} = \frac{1}{\theta} \bn^{\prime}_t$.  \par
The last step is to obtain the result for any orthonormal frame $\{f_{i, 
0}\}$ in $\Pi_0$.  
There is an orthogonal transformation $A$ so that $A(e_{i, 0}) = f_{i, 0}$.  
If we express $f_{i, 0} = \sum_{j = 1}^{n-1} a_{i, j} e_{j, 0}$.  Then, we can 
define vector fields $\tilde f_{i, t} = \sum_{j = 1}^{n-1} a_{i, j} \tilde 
e_{j, t}$.  Since the $\tilde f_{i, t}$ are constant linear combinations of 
Lorentzian parallel vector fields,and hence are Lorentzian parallel 
themselves.  As they are obtained by an orthogonal transformation of an 
orthonormal frame field, they also form an orthonormal frame field. 
\end{proof}
\par
\subsection*{Interpolating between Orthonormal Frames} 
\par
Now we consider given frames $\{e_{i, 0}\}$ in $\Pi_0$ and $\{f_{i, 1}\}$ 
in $\Pi_1$, such that $\{\bn_0, e_{1, 0}, \dots , e_{n-1, 0}\}$ and 
$\{\bn_1, f_{1, 0}, \dots , f_{n-1, 0}\}$ have the same orientation (which 
we may assume are positive.  We may first construct the Lorentzian 
parallel family of orthonormal frames $\{e_{i, t}\}$.  Then, the smoothly 
varying family $\{\bn_t, e_{1, t}, \dots , e_{n-1, t}\}$ will retain positive 
orientation for each $t$.  Hence, $\{e_{i, 1}\}$ and $\{f_{i, 1}\}$ have the 
same orientation.  Thus, there is an orthogonal transformation $B$ of 
$\Pi_1$ such that $B(e_{i, 1}) = f_{i, 1}$ and $\det(B) = 1$.  Again we may 
represent $B$ using the basis $\{e_{i, 1}\}$ by an orthogonal matrix $b_{i, 
j}$.  As $\det(B) = 1$, there is a one parameter family $\exp(s E)$ so that 
$\exp( E) = B$ for a skew symmetric matrix $E$.  Then, given a smooth 
nondecreasing function $\varphi : [0, 1] \to [0, 1]$ with $\varphi(0) = 0$ 
and $\varphi(1) = 1$, we can modify the Lorentzian parallel family $\{e_{i, 
t}\}$ by 
$\{ \exp(\varphi(t) E)(e_{i, t})\}$, which is a family of orthonormal 
frames.  In this family we see that the \lq\lq total amount of 
twisting\rq\rq\, from Lorentzian parallel family is given by the 
orthogonal transformation $B$ (or skew-symmetric matrix $E$).  The 
introduction of the twisting in the family is given by the function 
$\varphi$.   
\par
\begin{Example}[Planes in $\R^3$]
\label{Exam6.5}
\par
In the case of planes $\Pi_0$ and $\Pi_1$ in $\R^3$ with $\bn_0 \neq \pm 
\bn_1$ and $\bn_0 \cdot \bn_1 = \cos(\theta)$, we can easily construct 
the family of Lorentzian parallel frames by letting $e_{2, t}$ be a 
constant unit vector field in the direction of $\bn_0 \times \bn_1$, and 
$\tilde e_{1, t} = \frac{1}{\theta} \bn^{\prime}_t$ for the Lorentzian 
geodesic flow $\gg(t) = (\bn_t, c_t)$ from $\Pi_0$ to $\Pi_1$.  Then, 
$\{e_{1, t}, e_{2, t}\}$ gives a Lorentzian parallel family of frames.  \par
 If $e^{\prime}_1, e^{\prime}_2$ is another frame for $\Pi_0$ with the 
same orientation as $e_{1, 0}, e_{2, 0}$, then there is a rotation with 
rotation matrix $R$ so that $Re_{i, 0} = e^{\prime}_i$.  Then $\{R e_{1, t}, 
R e_{2, t}\}$ gives a Lorentzian parallel family of frames beginning 
with $e^{\prime}_1, e^{\prime}_2$.   Furthermore, if $\{ f_1, f_2\}$ is a 
positively oriented frame for $\Pi_1$, then there is a rotation matrix 
$S_{\gw}$ by an angle $\gw$ so that $S_{\gw}R e_{i, 1} = f_i$.  Then, for $\gw(t)$, 
$0 \leq t \leq 1$, a nondecreasing smooth function from $0$ to $1$, the 
family of rotations $S_{\gw(t)}$ provides an interpolating family 
$\{S_{\gw(t)}R e_{1, t}, S_{\gw(t)}R e_{2, t}\}$ from $\{e_{1, 0},  e_{2, 
t}\}$ to $\{ f_1, f_2\}$.  The flexibility in the choice of $\gw(t)$ allows 
for many criterion to be included in choosing the interpolation. \par
\begin{Remark}[Interpolation for Modeling Generalized Tube Structures]
\label{Rem6.5}
Generalized tube structures for a region $\gW$ can be modeled as a disjoint 
union $\gW = \cup_t \gW_t$ of planar regions $\gW_t \subset \Pi_t$ for a family 
of hyperplanes $\{\Pi_t\}$ along a central curve $\gg(t)$.  The geometric properties 
and structure of the tube can be computed using a smoothly varying family of 
frames $e_{j, t}$ for $\{\Pi_t\}$ (see e.g. \cite{D2} and \cite{D3}).  
	This is used in \cite{MZW} for the $3$-dimensional modeling of the human 
colon, where normal planes to an identified central curve are modified in high 
curvature regions to form a Lorentzian geodesic, and the family of frames with 
minimal twisting in the sense of Example \ref{Exam6.2a} are extended to a 
Lorentzian parallel family of frames in the modified family of planes.  This structure 
can then be deformed in various ways for better visualization.
\end{Remark}
\par 
\end{Example}
\begin{Example}[Family of Normal Planes to a Curve in $\R^3$]
\label{Exam6.6}
\par
A second situation is for a regular unit speed curve $\ga(t)$ in $\R^3$ 
with $\gk \neq 0$.  Then there is the Frenet frame $\{ \bT, \bN, \bB\}$ 
defined along $\ga(t)$.  Then, $\{\bN, \bB\}$ provides a family of 
orthonormal frame for the planes $\Pi_t$ passing through and orthogonal 
to $\ga(t)$.  The Lorentzian map for the family of planes is given by 
$\gg(t) = (\bT(t), \ga\cdot \ga^{\prime}\bge)$.  Then,  
$$  \gg^{\prime\prime}(t) = \gk^{\prime} \bN + \gk(-\gk T + \tau B), 
\frac{d^2 }{d^2t}(\ga\cdot \ga^{\prime})\bge )  $$
Thus, for this family to be a Lorentzian geodesic family of planes, 
$\gg^{\prime\prime}(t)$ must be Lorentzian orthogonal to $\gL^{4}$.  For 
this, we must have that the first term is a multiple of $\bT$, which 
implies $\gk^{\prime}, \tau \equiv 0$.  Thus, $\ga$ is a plane curve with 
constant curvature $\gk$, so it is a portion of a circle and $\ga\cdot 
\ga^{\prime} \equiv 0$.  Then, $B$ is constant so it is Lorentzian parallel, 
and $\gg^{\prime\prime}(t) = (-\gk^2 T, 0 \bge) = -\gk (-\gk T, 0 \bge)$, and 
so it follows that $(\bN, 0\bge)$ is Lorentzian parallel.  Hence, $\{ \bN, 
\bB\}$ is a Lorentzian parallel family of orthonormal frames.  We 
summarize this with the following
\begin{Proposition}
\label{Prop6.7}
If $\ga(t)$ is a regular unit speed curve in $\R^3$ with $\gk \neq 0$, then 
the family of normal planes $\Pi_t$ to $\ga(t)$ is a Lorentzian geodesic 
family iff $\ga(t)$ is a portion of a circle.  In this case the Frenet vector 
fields $\{\bN, \bB\}$ forms a Lorentzian parallel family of orthonormal 
frames in $\Pi_t$. 
\end{Proposition}
\end{Example}
\vspace{2ex}
\par
\section{\bf Dual Varieties and Singular Lorentzian Manifolds}
\label{Sec:dual.var} 
\par
Before continuing with the analysis of the geodesic flow in  $\gL^{n+1}$ 
and the induced flow between hypersurfaces in $\R^n$, we first explain 
the relation of the Lorentzian map with a corresponding map to the dual 
projective space.
\subsection*{Relation with the Dual Variety} \par
Suppose that $M \subset \R^n$ is a smooth hypersurface.  There is a 
natural way to associate a corresponding \lq\lq dual variety\rq\rq 
$M^{\vee}$ in the dual projective space $\R P^{n\, \vee}$ (which consists 
of lines through the origin in the dual space $\R^{n+1\, *}$).  Given a 
hyperplane $\Pi \subset \R^n$,  it is defined by an equation $\sum_{i 
=1}^{n} a_i x_i = b$.  We associate the linear form $\ga : \R^{n+1} \to \R$ 
defined by $\ga(x_1, \dots , x_{n+1}) = \sum_{i =1}^{n} a_i x_i - bx_{n+1}$.  
As the equation for $\Pi$ is only well defined up to multiplication by a 
constant, so is $\ga$, which defines a unique line in $\R^{n+1\, *}$.  This 
then defines a {\it dual mapping} $\gd : M \to \R P^{n\, \vee}$, sending $x 
\in M$ to the dual of $T_xM$.  \par
In the context of algebraic geometry in the complex case, this map actually 
extends to a dual map for a smooth codimension $1$ algebraic subvariety 
$M \subset \C P^{n}$, and then the image $M^{\vee} = \gd (M)$ is again a 
codimension $1$ algebraic subvariety of $\C P^{n\, \vee}$.  There is an 
inverse dual map $\gd^{\vee}$ for smooth codimension $1$ algebraic 
subvarieties of $\C P^{n\, \vee}$ to $\C P^{n}$ defined again using the 
tangent spaces.  Hence, $\gd^{\vee} : M^{\vee} \to \C P^{n}$.  It is only 
defined on smooth points of $M^{\vee}$ (which may have singularities); 
however it extends to the singular points of $M^{\vee}$ and its image is 
the original $M$.  \par
In our situation, we are working over the reals and moreover $M$ will not 
be defined algebraically.  Hence, we need to determine what properties both 
$\gd$ and $M^{\vee}$ have.  We also will explain the relation with the 
Lorentz map.  
\par
\subsection*{Legendrian Projections} 
\par
Given $M$, we let $P(\R^{n+1\,*})$ denote the projective bundle over $\R^n$ 
given by $\R^{n} \times \R P^{n\, \vee}$, where as earlier $\R P^{n\, \vee}$ denotes 
the dual projective space.  Then, we have an embedding $i : M \to 
P(\R^{n+1\,*})$, where $i(x) = (x, <\ga_x>)$, with $\ga_x$ the linear form 
associated to $T_xM$ as above.  We let $\tilde M = i(M)$.  There is a 
projection map $\pi : P(\R^{n+1\,*}) \to \R P^{n\, \vee}$.  Then, by results 
of Arnol\rq d \cite{A1}, $\pi$ is a Legendrian projection, and for generic $M$, 
$\tilde M$ is a generic Legendrian submanifold of $P(\R^{n+1\,*})$ and the 
restriction $\pi | \tilde M : \tilde M \to \R P^{n\, \vee}$ is a generic 
Legendrian projection.  This composition $\pi | \tilde M \circ i$ is exactly 
$\gd$.  Hence, the properties of $\gd$ are exactly those of the Legendrian 
projection.  In particular, the singularities of $M^{\vee} = \pi (\tilde M)$ are 
generic Legendrian singularities, which are the singularities appearing in 
discriminants of stable mappings, see \cite{A1} or \cite[Vol 2]{AGV}.  
\par
In the case of surfaces in $\R^3$, these are: cuspidal edge, a swallowtail, 
transverse intersections of two or three smooth surfaces, and the 
transverse intersection of a smooth surface with a cuspidal edge  (as 
shown in Fig. \ref{fig.1}).  The characterization of these singularities implies 
that as we approach a singular point from one of the connected components, 
then there is a unique limiting tangent plane, and in the case of the cuspidal 
edge or swallowtail, the limiting tangent plane is the same for each component.  
Hence, for generic smooth hypersurfaces $M \subset \R^n$, the inverse 
dual map $\gd^{\vee}$ extends to all of $M^{\vee}$, and again will have 
image $M$.  
\par

\begin{figure}[ht]
\centerline{\includegraphics[width=10cm]{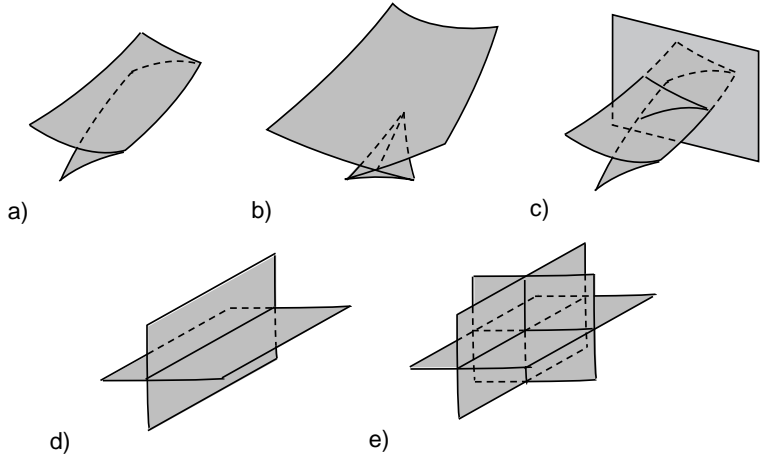}}
\caption{\label{fig.1}  Generic Singularities for Legendrian projections of 
Legendrian surfaces:  
a) cuspidal edge, b) swallowtail, c) transverse intersection of cuspidal 
edge and smooth surface,  d) transverse intersection of two smooth 
surfaces, and  e) transverse intersection of three smooth surfaces.}
\end{figure} 
 \par

\flushpar

\par
\vspace{2ex}
Finally, we remark about the relation between the dual variety $M^{\vee}$ 
and the image $M_{\cL} = \cL (M)$ (or $M_{\tilde \cL} = \tilde \cL (M)$).  
To do so, we introduce a mapping involving $\R P^{n\, \vee}$ and $\tilde 
\cT^n$.    
In $\R P^{n\, \vee}$, there is the distinguished point $\infty = <(0, \dots, 0, 1)>$.  
On $\R P^{n\, \vee} \backslash \{ \infty\}$, we may take a point 
$<(y_1, \dots , y_n, y_{n+1})>$, and normalize it by 
$$(y_1^{\prime}, \dots , y_n^{\prime}, y_{n+1}^{\prime}) = c\cdot (y_1, 
\dots , y_n, y_{n+1}), \quad \mbox{where }\, c = (\sum_{i = 1}^{n} 
y_i^2)^{-\frac{1}{2}}.$$  
Then, $\bn_y = (y_1^{\prime}, \dots , y_n^{\prime})$ is a unit vector.  We 
then define a map $\nu : \R P^{n\, \vee} \backslash \{ \infty\} \to \tilde 
\cT^n$ sending $<(y_1, \dots , y_n, y_{n+1})>$ to $(\bn_y, y_{n+1}^{\prime} 
\bge)$.  This is only well-defined up to multiplication by $-1$, which is why we 
must take the equivalence class in the pair of points. If we are on a region 
of $\R P^{n\, \vee} \backslash \{ \infty\}$ where we can smoothly choose a 
direction for each line corresponding to a point in $\R P^{n\, \vee}$, then as 
for the case of the Lorentzian mapping, we can give a well-defined map to 
$\cT^n$.  This will be so when we consider $M^{\vee}$ for the oriented case.  
In such a situation, when the smooth hypersurface $M$ has a smooth unit 
normal vector field $\bn$, it provides a positive direction in the line of 
linear forms vanishing on $T_xM$.  \par 
Then, we have the following relations.  
\begin{Lemma}
\label{Lem3.3}
The smooth mapping $\tilde \nu : \R P^{n\, \vee} \backslash \{ \infty\} \to 
\tilde \cT^n$ is a diffeomorphism.
\end{Lemma}
Second, there is the relation between the duality map  $\gd$ and the 
Lorentz map $\tilde \cL$ (or $\cL$).  
\begin{Lemma}
\label{Lem3.4}
If $M \subset \R^n$ is a smooth hypersurface, then the diagram 
(\ref{tag3.2}) commutes, i.e. 
$\tilde \nu \circ \gd = \tilde \cL$.  If, in addition, $M$ has a smooth unit 
normal vector field $\bn$, then there is the oriented version of diagram 
(\ref{tag3.2}),  $\nu \circ \gd = \cL$.  
\begin{eqnarray}
\label{tag3.2}
M & \overset{\gd}{\longrightarrow} & \R P^{n\, \vee}  \notag\\
   &     \underset{\tilde \cL\,\,}{\searrow}      &   \downarrow \tilde \nu  
\\
   &       &     \tilde \cT^n  \notag  
\end{eqnarray}
\end{Lemma}
\par

\vspace{2ex}
As a consequence of these Lemmas and our earlier discussion about the 
singularities of $M^{\vee}$, we conclude that $M_{\tilde \cL}$ (or 
$M_{\cL}$) have the same singularities.  Thus, we may suppose they are 
generic Legendrian singularities.  
\begin{Remark}
Although by Lemma \ref{Lem3.3} $\R P^{n\, \vee} \backslash 
\{ \infty\}$ is diffeomorphic to $\tilde \cT^n$, the first space has a natural 
Riemannian structure while on $\tilde \cT^n$ we have a Lorentzian metric.  
Hence, $\tilde \nu$ is not an isometry and does not map geodesics
 to geodesics.
\end{Remark}

\begin{proof}[Proof of Lemma \ref{Lem3.3}]
There is a natural inverse to $\tilde \nu$ defined as follows:  
If $z = (\bn, c\bge)$ and $\bn = (a_1, \dots , a_n)$, then we map $z$ to 
$<(a_1, \dots , a_n, -c)>$.  We note that replacing $z$ by $-z$ does not 
change the line $<(a_1, \dots , a_n, -c)>$.  This gives a well-defined 
smooth 
map $\tilde \cT^n \to \R P^{n\, \vee} \backslash \{ \infty\}$ which is 
easily 
checked to be the inverse of $\tilde \nu$.
\end{proof} 
\par
\begin{proof}[Proof of Lemma \ref{Lem3.4}]
If $T_xM$ is defined by $\bn\cdot \bx = c$ with $\bn = (a_1, \dots , a_n)$, 
then $\gd (x) = < (a_1, \dots , a_n, -c) >$.  Then, as $\| \bn\| = 1$, $\tilde 
\nu ( < (a_1, \dots , a_n, -c) >) = (a_1, \dots , a_n, c, c) = (\bn,  c \bge)$, 
which is exactly $\cL(x)$.
\end{proof} 
\par
\subsection*{Inverses of the Dual Variety and Lorentzian Mappings} 
\par
We consider how to invert both $\gd$ and $\tilde \cL$.  We earlier 
remarked 
that in the complex algebraic setting, the inverse to $\gd$ is again a dual 
map $\gd^{\vee}$.  As $\tilde \nu$ is a diffeomorphism, and diagram 
\ref{tag3.2} commutes, inverting $\gd$ is equivalent to inverting $\tilde 
\cL$.  Also, constructing an inverse is a local problem, so we may as well 
consider the oriented case.  
\begin{Proposition}
\label{Prop3.5}
Let $M \subset \R^n$ be a generic smooth hypersurface with a smooth unit 
normal vector field $\bn$.  Suppose that the image $M_{\cL}$ under $\cL$ 
is a smooth submanifold of $\cT^n$.  Then, $M$ is obtained as the envelope 
of the collection of hyperplanes defined by $\bn \cdot \bx = c$ for 
$\cL (x) = (\bn, c\bge)$.
\end{Proposition}
\par
\begin{proof}[Proof of Proposition \ref{Prop3.5}]
We consider an $(n-1)$-dimensional submanifold of $\cT^n$ parametrized 
by $u \in U$ given by $(\bn (u), c(u)\bge)$.  The collection of hyperplanes are 
given by $\Pi_u$ defined by $F(\bx, u) = \bn (u)\cdot \bx - c(u) = 0$. 
Then, the envelope is defined by the collection of equations 
$F_{u_i} = 0, i = 1, \dots , n-1$ and $F = 0$.  This is the system of linear equations 
\begin{equation}
\label{Eqn3.8}
i)\,\,\, \bn (u)\cdot \bx = c(u) \qquad \mbox{and }\, ii)\, \,\, 
\bn_{u_i} (u)\cdot \bx = c_{u_i}(u), \,\, i = 1, \dots , n-1  
\end{equation}

A sufficient condition that there exist for a given $u$ a unique solution to 
the system of linear equations in $\bx$ is that the vectors 
$\bn, \bn_{u_1}, \dots , \bn_{u_{n-1}}$ are linearly independent.  Since 
$\bn_{u_i} = - S(\pd{ }{u_i})$, for $S$ the shape operator for $M$, linear 
independence is equivalent to $S$ not having any $0$-eigenvalues.  
Thus, $\bx$ is not a parabolic point of $M$.  For generic $M$, the set of 
parabolic points is a stratified set of codimension $1$ in $M$.  Thus, off 
the image of this set, there is a  unique point in the envelope.  \par
Also, if we differentiate equation (\ref{Eqn3.8})-i) with respect to $u_i$ 
we 
obtain 
\begin{equation}
\label{Eqn3.9}
\bn_{u_i} (u)\cdot \bx \, + \, \bn (u)\cdot \bx_{u_i} \,\,  = \,\, c_{u_i}(u)
\end{equation}
Combining this with (\ref{Eqn3.8})-ii), we obtain 
\begin{equation}
\label{Eqn3.10}
\bn (u)\cdot \bx_{u_i} \,\,  = \,\, 0,
\end{equation}
and conversely, (\ref{Eqn3.10}) for $i = 1, \dots , n-1$ and (\ref{Eqn3.9}) 
imply (\ref{Eqn3.8})-ii).  Thus, if we choose a local parametrization of 
$M$ 
given by $\bx(u)$, then as $\bx(u)$ is a point in its tangent space, it 
satisfies (\ref{Eqn3.8})-i), and hence (\ref{Eqn3.9}), and also $\bn$ being a 
normal vector field  implies that (\ref{Eqn3.10}) is satisfied for all $i$.  
Thus, (\ref{Eqn3.8})-ii) is satisfied.  Hence, $M$ is part of the envelope.  
Also, for generic points of $M$, by the implicit function theorem, the set 
of 
solutions of (\ref{Eqn3.8}) is locally a submanifold of dimension $n-1$.  
Hence, in a neighborhood of these generic points of $M$, the envelope is 
exactly $M$.  Hence, the closure of this set is all of $M$ and still consists 
of 
solutions of (\ref{Eqn3.8}).  Thus, we recover $M$.  \par
Second, to see that the equations (\ref{Eqn3.8}) describe the inverse of 
the 
dual mapping, we note by Lemmas \ref{Lem3.3} and \ref{Lem3.3} that 
$\tilde \nu$ is a diffeomorphism, $\gd^{-1} = \tilde \cL^{-1} \circ \tilde 
\nu$, and the preceding argument gives the local inverse to $\tilde \cL$.
\end{proof} 
\par

\section{\bf Sufficient Condition for Smoothness of Envelopes}
\label{Sec:smoot.lev} 
\par

To describe the induced \lq\lq geodesic flow\rq\rq between hypersurfaces 
$M_0$ and $M_1$ in $\R^n$, we will use the Lorentzian geodesic flow in 
$\cT^n$ and then find the corresponding flow by applying an inverse to 
$\cL$.  
We begin by constructing the inverse for a $(n-1)$-dimensional manifold 
in $\cT^n$ parametrized by $(\bn (u), c(u) \bge)$, where 
$u = (u_1, \dots , u_{n-1})$.  We determine when the associated family of 
hyperplanes $\Pi_u = \{ \bx \in \R^n: \bn(u) \cdot \bx = c(u)\}$ has as an 
envelope a smooth hypersurface in $\R^n$.  \par
We introduce a family of vectors in $\R^{n+1}$ given by $\tilde \bn(u) = 
(\bn(u), -c(u))$.  We also denote $\pd{\tilde \bn}{u_i}$ by $\tilde 
\bn_{u_i}$.  
Next we consider the $n$-fold cross product in $\R^{n+1}$, denoted by 
$v_1 \times v_2 \times \cdots \times v_n$, which is the vector in 
$\R^{n+1}$ whose $i$-th coordinate is $(-1)^{i+1}$ times the $n \times n$ 
determinant obtained from the entries of $v_1, \dots, v_n$ by removing 
the 
$i$-th entries of each $v_j$.  Then, for any other vector $v$, 
$$ v \cdot (v_1 \times v_2 \times \cdots \times v_n) \,\, = \,\, \det(v, 
v_1, \dots, v_n) $$
We let 
$$\tilde \bh  \,\, = \, \, \tilde \bn \times \tilde \bn_{u_1} \times \cdots 
\times \tilde \bn_{u_{n-1}}  $$
We let $H(\tilde \bn)$ denote the $(n-1) \times (n-1)$ matrix of vectors 
$\tilde \bn_{u_i \, u_j}$. Then we can form $H(\tilde \bn)\cdot \tilde 
\bh$ to 
be the $(n-1) \times (n-1)$ matrix with entries $\tilde \bn_{u_i \, 
u_j}\cdot 
\tilde \bh$.  
Then, there is the following determination of the properties of the 
envelope 
of $\{ \Pi_u\}$.
\begin{Proposition}
\label{Prop5.1}
Suppose we have an $(n-1)$-dimensional manifold in $\cT^n$ parametrized 
by $(\bn (u), c(u) \bge)$, where $u = (u_1, \dots , u_{n-1})$.  We let 
$\{ \Pi_u\}$ denote the associated family of hyperplanes.  Then, the 
envelope of $\{ \Pi_u\}$ has the following properties.
\begin{itemize}
\item[i)]  There is a unique point $\bx_0$ on the envelope corresponding to 
$u_0$ provided $\bn(u_0), \bn_{u_1}(u_0), \dots, \bn_{u_{n-1}}(u_0)$ are 
linearly independent.  Then, the point is the solution of the system of equations (\ref{Eqn3.8}).
\item[ii)]  Provided i) holds, the envelope is smooth at $\bx_0$ provided 
$H(\tilde \bn)\cdot \tilde \bh$ is nonsingular for $u = u_0$.
\item[iii)]  Provided ii) holds, the normal to the surface at $\bx_0$ is 
$\bn(u_0)$ and $\Pi_{u_0}$ is the tangent plane at $\bx_0$.
\end{itemize}
\end{Proposition}
\begin{proof}[Proof of Proposition \ref{Prop5.1}]
We use the line of reasoning for Proposition \ref{Prop3.5}.  the condition 
that a point $\bx_0$ belong to the envelope of $\{ \Pi_u\}$ is that it satisfy 
the system of equations (\ref{Eqn3.8}).  A sufficient condition that these 
equations have a unique solution for $u = u_0$ is exactly that $\bn(u_0), 
\bn_{u_1}(u_0), \dots, \bn_{u_{n-1}}(u_0)$ are linearly independent.   \par 
Furthermore, if this is true at $u_0$ then it is true in a neighborhood of 
$u_0$.  Thus, we have a unique smooth mapping $\bx(u)$ from a 
neighborhood of $u_0$ to $\R^n$.  By the argument used to deduce 
(\ref{Eqn3.10}), we also conclude
\begin{equation}
\label{Eqn5.2}
\bn (u)\cdot \bx_{u_i} \,\,  = \,\, 0,  \qquad i = 1, \dots, n-1
\end{equation}
Hence, if $\bx(u)$ is nonsingular at $u_0$, then $\bn(u_0)$ is the normal 
vector to the envelope hypersurface at $\bx_0$, so the tangent plane is 
$\Pi_{u_0}$.  Thus iii) is true. \par
It remains to establish the criterion for smoothness in ii).  As earlier 
mentioned the envelope in the neighborhood of a point $\bx_0$ is the 
discriminant of the projection of $V = \{(\bx, u) : F(\bx, u) = \bn (u)\cdot 
\bx - c(u) = 0\}$ to $\R^n$.  It is a standard classical result that at a point 
$(\bx_0, u_0) \in V$, which projects to an envelope point $\bx_0$, the 
envelope is smooth at $\bx_0$ provided $(\bx_0, u_0)$ is a regular point 
of $F$ (so $V$ is smooth in a neighborhood of $(\bx_0, u_0)$) and the partial 
Hessian $(\pd{^2F}{u_i\, u_j}(\bx_0, u_0))$ is nonsingular.  For our 
particular $F$ this Hessian becomes $H(\bn)\cdot \bx_0 - H(c)$, where 
$H(\bn)$ is the $(n-1) \times (n-1) $ matrix of vectors $(\bn_{u_i \, u_j})$, and $H(\bn)\cdot \bx_0$ denotes the $(n-1) \times (n-1)$ matrix whose entries 
are $\bn_{u_i \, u_j} \cdot \bx_0$.  This can also be written $H(\tilde \bn)\cdot \tilde \bx_0$, where $\tilde \bx_0$ is the extension of $\bx_0$ to $\R^{n+1}$ by adding 
$1$ as the $n+1$-st coordinate. \par  
Now $\bx_0$ is the unique solution of the system of linear equations 
(\ref{Eqn3.8}).  This solution is given by Cramer\rq s rule.   Let $N(u_0)$ 
denote the $n \times n$ matrix with columns $\bn(u_0), \bn_{u_1}(u_0), 
\dots , \bn_{u_{n-1}}(u_0)$.  By Cramer\rq s rule, if we multiply 
$\tilde \bx_0$ by $\det (N(u_0))$ we obtain $(-1)^n \tilde \bh$.  Thus, 
multiplying $H(\tilde \bn)\cdot \tilde \bx_0$ by $\det (N(u_0))$ yields 
$(-1)^n H(\tilde \bn)\cdot \tilde \bh$.  Hence, the nonsingularity of 
$H(\tilde \bn)\cdot \tilde \bh$ implies that of $(\pd{^2F}{u_i\, u_j}(\bx_0, u_0))$.
\end{proof}

\par
Although Proposition \ref{Prop5.1} handles the case of a smooth manifold 
in $\cT^n$, we saw in \S \ref{Sec:dual.var} that usually the image in 
$\cT^n$ of a generic 
hypersurface $M$ in $\R^n$ will have Legendrian singularities and the 
image itself is a Whitney stratified set $\tilde M$.  Next, we deduce the 
condition ensuring that the  envelope is smooth at a singular point 
$\bx_0$.  \par
Because $\tilde M$ has Legendrian singularities, it has a special property.  
To expain it we use a special property which holds for certain Whitney 
stratified sets.
\begin{Definition}
\label{Def5.2}  An $m$-dimensional Whitney stratified set $M \subset \R^k$ 
has the {\em Unique Limiting Tangent Space Property} (ULT property) if 
for any $x \in M_{sing}$, a singular point of $M$, there is a unique 
$m$-plane $\Pi \subset \R^k$ such that for any sequence $\{ x_i \}$ of 
smooth points in $M_{reg}$ such that $\lim x_i = x$, we have 
$\lim T_{x_i}M = \Pi$
\end{Definition}
\par
\begin{Lemma}
\label{Lem5.3}
For a generic Legendrian hypersurfaces $M \subset \R^n$, if $z \in \tilde M$, then 
$\tilde M$ can be locally represented in a neighborhood of $z$ as a finite 
transverse union of $(n-1)$-dimensional Whitney stratified sets $Y_i$ 
each 
having the ULT property.
\end{Lemma}
Transverse union means that if $W_{i j}$ is the stratum of $Y_i$ 
containing 
$z$ than the $W_{i j}$ intersect transversally.  \par
\begin{proof} The Lemma follows because $\tilde M$ consists of generic 
Legendrian singularities, which are either stable (or topologically stable) 
Legendrian singularities.  These are either discriminants of stable 
unfoldings of multigerms of hypersurface singularities or transverse sections 
of such.  Such discriminants are transverse unions of discriminants of individual 
hypersurface singularities, each of which have the ULT property by a 
result of Saito [Sa].  This continues to hold for transverse sections.  
\end{proof}
We shall refer to these as the {\it local components} of $\tilde M$ in a 
neighborhood of $z$.
\par
There is then a corollary of the preceding.
\begin{Corollary}
\label{Cor5.4}
Suppose that $\tilde M$ is an $(n-1)$--dimensional Whitney stratified set in 
$\cT^n$ such that: at every smooth point $z$ of $\tilde M$, the hypotheses 
of Proposition \ref{Prop5.1} holds; and at all singular points $\tilde M$ is 
locally the finite union of Whitney stratified sets $Y_i$ each having the ULT 
property.  Then, 
\begin{itemize}
\item[i)] The envelope of $M$ of $\tilde M$ has a unique point $x \in M$ for 
each $z \in \tilde M_{reg}$, and $M$ is smooth at all points corresponding 
to 
points in $\tilde M_{reg}$.  
\item[ii)]  At each singular point $z$ of $\tilde M$, there is a point in $M$ 
corresponding to each local component of $\tilde M$ in a neighborhood of 
$z$. 
\end{itemize}
\end{Corollary}
\begin{proof}
First, if $z \in \tilde M_{reg}$ and satisfies the conditions of Proposition 
\ref{Prop5.1}, then there is a unique envelope point corresponding to $z$ 
and the envelope is smooth at that point.  \par
Second, via the isomorphism $\tilde \nu$ and the commutative diagram 
(3.1), the envelope construction corresponds to the inverse $\gd^{\vee}$ of 
$\gd$ (or rather a local version since we have an orientation).  Under the 
isomorphism $\tilde \nu$, for each point $z \in \tilde M_{sing}$ there 
corresponds a unique point in the envelope for each local component of 
$\tilde M$ containing $z$. It is obtained as $\gd^{\vee}$ applied to the 
unique 
limiting tangent space of $z$ associated to the local component in $\tilde 
M_{reg}$.
\end{proof}
\section{\bf Induced Geodesic Flow between Hypersurfaces}
\label{Sec:ind.flow} 
\par
We can bring together the results of the previous sections to define the 
Lorentzian geodesic flow between two smooth generic hypersurfaces with 
a correspondence.  We denote our hypersurfaces by $M_0$ and $M_1$ and 
let $\chi : M_0 \to M_1$ be a diffeomorphism giving the correspondence.  
Note that we allow the hypersurfaces to have boundaries.  \par 
We suppose that both are oriented with unit normal vector fields $\bn_0$ 
and $\bn_1$.  We also need to know that they have a \lq\lq local relative 
orientation\rq\rq.
\begin{Definition}
\label{Def6.1}
We say that the oriented manifolds $M_0$ and $M_1$, with unit normal 
vector fields $\bn_0$ and $\bn_1$, and with correspondence $\chi : M_0 
\to M_1$ are {\em relatively oriented} if there is a smooth function 
$\theta(x) : M_0 \to (-\pi, \pi)$ such that $\bn_0(x) \cdot \bn_1(\chi (x)) = 
\theta(x)$ for all $x \in M_0$.  
\end{Definition} 
\par 
An example of a Lorentzian geodesic flow between curves in $\R^2$ is illustrated in Figure \ref{fig.10}. 
\par
\begin{figure}[ht]
\centerline{\includegraphics[width=6cm]{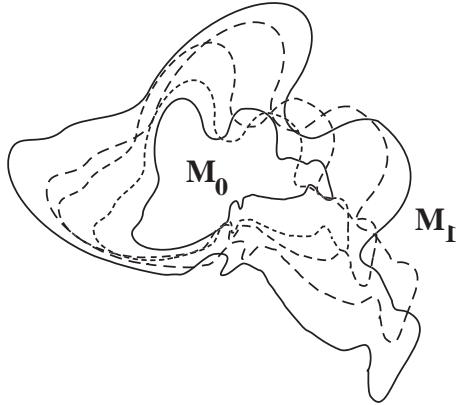}}
\caption{\label{fig.10}  A nonsingular Lorentzian Geodesic Flow between the 
curve $M_0$ in $\R^2$ and the curve $M_1$, which was obtained from $M_0$ 
via a composition of a rigid motion and a homothety.  The correspondence is 
given by the combined transformations, and then the relative orientation is a 
constant angle.  As remarked in b) of Figure \ref{fig.2} there does not exist a 
nonsingular geodesic flow between $M_0$ and $M_1$ in $\R^2$.}
\end{figure} 
\par
If the preceding example in Figure \ref{fig.10} is slightly perturbed, then the existence 
of a nonsingular Lorentzian flow is guaranteed by the next theorem.
\par
\begin{Thm}[ {\bf Existence, Smooth Dependence and Stability of 
Lorentzian 
Geodesic Flows} ]\hfill
\label{Thm6.1}
\par  
Suppose smooth generic hypersurfaces $M_0$ and $M_1$ are oriented by 
smooth unit normal vector fields $\bn_i, i = 0, 1$ and are relatively 
oriented by $\theta$ for the diffeomorphism $\chi : M_0 \to M_1$.  
\begin{enumerate}
\item {\em (Existence and Smoothness:) } Then for the given relative 
orientation, is a smooth Lorentzian 
geodesic flow $\psi_t : M_0 \times [0, 1] \to \cT^n$ between $M_0$ and 
$M_1$ given by \eqref{Eqn8.4a}.  \par
\item {\em (Stability:) } There is a neighborhood $\cU$ of $\chi$ in $\Diff 
(M_0, M_1)$ (for the $C^{\infty}$--topology) such that if $\chi^{\prime} 
\in \cU$, then $M_0$ and $M_1$ are relatively oriented for $\chi^{\prime}$ 
and the map $\Psi : \cU \to C^{\infty}(M_0 \times [0, 1], \cT^n)$ mapping 
$\chi^{\prime}$ to the associated Lorentzian flow $\tilde \psi^{\prime}_t$ 
is continuous.  \par
\item {\em (Smooth Dependence:) } Let $\chi_s : M_{0\, s}  \to M_{1\, s}$ 
be a smooth family of diffeomorphisms between smooth families of 
hypersurfaces for $s \in S$, a smooth manifold (i.e. $M_{i\, s}$ is the 
image of $M_i \times S$ under a smooth family of embeddings) so that 
$M_{0\, s}$ and $M_{1\, s}$ are relatively oriented for $\chi_s$ for each 
$s$ by a smooth map $\theta(x, s) : M_{0\, s} \to (-\pi, \pi)$ in $(x, s)$.  Then, 
the family of Lorentzian Geodesic flows 
$\tilde \psi_{s, t} : M_0 \times S \times [0, 1] \to \cT^n$ between 
$M_{0\, s}$ and $M_{1\, s}$ is a smooth function of $(x, s, t)$.
\end{enumerate}
\end{Thm}  
\par 
\begin{proof} \par
Using the form of the Lorentzian geodesic flow given by Proposition 
\ref{prop4.1} we have the Lorentzian geodesic flow is defined by
\begin{equation}
\label{Eqn8.4a}
  \psi_t(x) \quad = \quad \gl(t, \theta(x))\, z_1(x) \, + \,  \gl(1-t, 
\theta(x))\, z_0(x)   
\qquad \mbox{for }   0 \leq t \leq 1
\end{equation}
Here $z_0(x) = (\bn_0(x), c_0(x))$ for $T_xM_0$ defined by $\bn_0(x) 
\cdot \bx = c_0(x)$, and $z_1(x) = (\bn_1(x), c_1(x))$ for $T_{\chi(x)}M_1$ 
defined by $\bn_1(x) \cdot \bx = c_1(x)$.  As $z_i(x)$ and $\theta(x)$ 
depend smoothly on $x \in M$ and $\gl(t, \theta)$ is smooth on $[0, 1] 
\times (-\pi, \pi)$, $\psi_t(x) $ is smooth in $(x, t)$.  
Hence, the Lorentzian flow is a smooth well-defined flow between 
$(\bn_0(x), c_0(x)\bge )$ and $(\bn_1(\chi(x)), c_1(\chi(x))\bge )$  
\par
For smooth dependence 3), we use an analogous argument.  We use 
\eqref{Eqn8.4a} but with $\theta(x)$ replaced by $\theta(x, s)$ and each 
$z_i(x)$ by $z_i(x, s) = (\bn_1(x, s), c_1(x, s))$ where $T_xM_{0\, s}$ is 
defined by $\bn_0(x, s) \cdot \bx = c_0(x, s)$ and $T_{\chi(x, s)}M_{1\, s}$ 
is defined by $\bn_1(x, s) \cdot \bx = c_1(x, s)$.  
\par 
Finally to establish the stability, given $\chi$ for which $M_0$ and $M_1$ 
are relatively oriented via the smooth function $\theta(x)$, we let 
$\gd(x)$ be a smooth nonvanishing function such that  $\gd(x) < 1/3 (\pi - 
\theta(x))$ and $\lim \gd(x) = 0$ as $x$ approaches any \lq\lq unbounded 
boundary component at $\infty$\rq\rq\, of $M_0$.  Then, as $(-\pi, \pi)$ is 
contractible there is a Whitney open neighborhood $\cU$ of $\chi$ such that 
if $\chi^{\prime} \in \cU$ 
then there is a smooth $\theta^{\prime} : M_0 \to  (-\pi, \pi)$ such that 
$\bn_0(x) \cdot \bn_1(\chi^{\prime}(x)) = \cos(\theta^{\prime}(x))$ and 
$| \theta^{\prime}(x) - \theta(x)| < \gd(x)$ for all $x \in M_0$.  Furthermore, 
$\theta^{\prime}$ depends continuously on $\chi^{\prime}$.  
Thus, the corresponding flow in \eqref{Eqn8.4a} defined by 
$\theta^{\prime}$ depends continuously on $\chi^{\prime}$.  
\par
Specifically, given $\chi^{\prime} \in \cU$, consider the mapping 
$\chi^{\prime}_{\cL} : M_0 \to \cT^n \times \cT^n$ defined by $x \mapsto 
((\bn_0(x), c_0(x)), (\bn_1(x), c_1(x)))$, where $(\bn_0(x), c_0(x))$ defines 
the tangent space $T_{x}M_0$ and $(\bn_1(x), c_1(x))$ defines the tangent 
space $T_{\chi^{\prime}(x)}M_1$.  Then, $\chi^{\prime}_{\cL}$ is defined 
using the first derivatives of the embeddings $M_i \subset \R^n$ and 
$\chi^{\prime}$ composed with algebraic operations.  Each such operation 
is continous in the Whitney $C^{\infty}$--topology and so defines a 
continuous map $\cL^{\prime} : \cU \to C^{\infty}(M_0, \cT^n \times 
\cT^n)$.  Lastly, the Lorentzian flow $\psi_t$ is defined by (\ref{Eqn3.4}), 
and is the composition of $\cL^{\prime}$ with algebraic operations 
involving the smooth functions 
$\gl(x, \theta)$, and is again continuous in the $C^{\infty}$--topology.  
Hence, the combined composition mapping $\chi{\prime} \to \psi_t$ is 
continuous in the $C^{\infty}$--topology.  \par 

\end{proof}
\begin{SimpRemark}
We note there are two consequences of 2) of Theorem \ref{Thm6.1}.  First, 
$M_0$ and $M_1$ may remain fixed, but the correspondence $\chi$ varies 
in a family.  Then the corresponding Lorentzian geodesic flows vary in a 
family.  Second, $M_0$ and $M_1$ may vary in a family with a 
corresponding varying correspondence, then the Lorentzian geodesic flow 
will also vary smoothly in a family.
\end{SimpRemark} 
\par
\subsection*{Nonsingularity of Level Hypersurfaces of Lorentzian Geodesic Flows in $\R^n$} \hfill
\par
It remains to determine when the corresponding Lorentzian geodesic flows 
in $\R^n$ will have analogous properties.  We give a criterion involving a 
generalized eigenvalue for a pair of matrices.  \par
We consider the vector fields on $M_0$, $\bn_0(x)$ and $\bn_1(\chi(x))$.  
For any vector field $\bn(x)$ on $M_0$ with values in $\R^n$, we let  
$N(x) = (\bn(x)\, | \, d\bn(x))$ be the $n \times n$ matrix with columns 
$\bn(x)$ viewed as a column vector and $d\bn(x)$ the $n \times (n-1)$ 
Jacobian matrix.  If we have a local parametrization $\bx(u)$ of $M_0$, 
then we may represent the vector field $\bn$ as a function of $u$, 
$\bn(u)$.  Then, $N(\bx(u))$ is the $n \times n$ matrix with columns 
$\bn(u), \bn_{u_1}(u), \dots , \bn_{u_{n-1}} (u)$.  
We denote this matrix for $\bn_0$ by $N_0(x)$, and that for 
$\bn_1(\chi(x))$ by $N_1(x)$ (or $N_0(u)$ and $N_1(\chi(u))$ if we have 
parametrized $M_0$).  By i) of Proposition \ref{Prop5.1} the Lorentzian 
geodesic flow in $\R^n$ will be well-defined provided the corresponding 
matrix $N_t(x)$ is nonsingular at all points of the flow $\tilde \psi_t(x)$.  We determine this by decomposing $N_t(x)$ into two parts. \par
First, there is the parametrized family of $n \times 
n$--matrices 
\begin{equation}
\label{Eqn9.1}
\widetilde{N}_t(x) \,\, \overset{def}{=} \,\, \gl(t, \theta)\, N_1(x) \, + \,  
\gl(1-t, \theta)\, N_0(x)  
\end{equation}
This captures the change resulting from the change in $x$.  To also capture the change resulting from that in $\theta$, we introduce a second matrix $\pd{\theta}{\bu}\, \bn_0$ whose first column equals the vector $0$ and whose $j+1$--th column is the vector 
$\pd{\theta}{u_j}\, \bn_0$, for $j = 1, \dots , n-1$.
Then, the nonsingularity criterion will be based on whether the pair of matrices 
$(\tilde N_t(x), \pd{\theta}{\bu}\, \bn_0)$ does not have a specific 
generalized eigenvalue.  \par 
Specifically we introduce one more function.  
$$\gs(x, \theta)\,\, = \,\, \pd{\gl}{\theta}(x, \theta) - x \cot(\theta) \gl(x, \theta)  $$
\par
Then, we compute for $\theta \neq 0$
\begin{equation}
\label{Eqn9.5}
 \pd{\gl(x, \theta)}{\theta} \,\, = \,\,   \frac{x\, \sin(\theta) 
\cos (x \theta)\, -\, \sin(x \theta) \cos \theta}{\sin^2\theta}  
\end{equation}
and $\pd{\gl(x, \theta)}{\theta}_{| \theta = 0} = 0$.
Using \eqref{Eqn9.5}, a direct calculation shows for all $0 \leq x \leq 1$.
$$\gs(x, \theta)\,\, = \,\,  \frac{\cos ((1-x)\theta) \sin(x\theta) \, -\, x\, 
\sin \theta }{\sin (x\theta)\, \sin\theta}  \,\, = \,\,  \frac{\cos ((1-
x)\theta) }{\sin\theta} \, - \,  \frac{x}{\sin (x\theta)}   $$
if $0 < |\theta | < \pi$; and 
$$  \gs(x, 0)\,\, = \,\,  0    $$
\par
We also define 
\begin{equation}
\label{Eqn9.2}
N_t^{\prime}(x) \,\, \overset{def}{=} \,\, \tilde N_t(x)  \, + \, \gs(t, 
\theta)\, \pd{\theta}{\bu}\, \bn_0
\end{equation}
Then, for any pair $(x, t)$, $N_t^{\prime}(x)$ is 
singular  iff $-\gs(t, \theta(x))$ is a generalized  eigenvalue 
for $(\tilde N_t(x), \pd{\theta}{\bu}\, \bn_0)$. 
\par
Consider the Lorentzian geodesic flow $\tilde \psi_t(x) = 
(\bn_t(x), c_t(x)\bge)$ between $\cL(x) = (\bn_0(x), c_0(x)\bge)$ and 
$\cL(\chi(x)) = (\bn_1(\chi(x)), c_1(\chi(x))\bge)$ for all $x \in M_0$.  
We let $\tilde M_t = \tilde \psi_t(M_0)$, and we let $M_t$ denote the 
envelope of $\tilde M_t$.   \par 
Then there are the following properties for the envelopes $M_t$ of the 
flow for all time $0 \leq t \leq 1$.  \par
\begin{Thm}
\label{Thm9.2}
Suppose smooth generic hypersurfaces $M_0$ and $M_1$ are oriented by 
smooth unit normal vector fields $\bn_i, i = 0, 1$ and are relatively 
oriented by $\theta(u)$.  
Let $\tilde \psi_t$ be the Lorentzian geodesic flow between $\tilde M_0$ 
and $\tilde M_1$ which is smooth.  If $M_t$ is the family of envelopes 
obtained from the flow $\tilde M_t =  \tilde \psi_t(\tilde M_0)$, then 
suppose that for each time $t$, $\tilde M_t$ has only generic Legendrian 
singularities as in \S \ref{Sec:dual.var} (as e.g. in Fig. \ref{fig.1}).  Then, 
\begin{enumerate}
\item  $M_t$ will have a unique point corresponding to $z = \tilde 
\psi_t(x)  \in \tilde M_t$ provided \eqref{Eqn9.2} is nonsingular.  
\item  The envelope $M_t$ will be smooth at points corresponding to a 
smooth point $ z \in \tilde M_t$ satisfying (\ref{Eqn9.2}) provided 
 $H(\tilde \bn_t(x))\cdot \tilde \bh_t(x)$ is nonsingular.  Here $\tilde 
\bh_t(x)$ is defined from $\tilde \bn_t(x)$ as in \S  \ref{Sec:smoot.lev}.
\item  At points corresponding to singular points $z \in \tilde M_t$, there 
is a unique point on $M_t$ for each local component of $\tilde M$ in a 
neighborhood of $z$.  This point is the unique limit of the envelope points 
corresponding to smooth points of the component of $\tilde M_t$ 
approaching $z$.  
\end{enumerate}
\end{Thm}
\par
\begin{Remark}
We observe that as a result of Theorem \ref{Thm9.2}, we can remark about the 
uniqueness of the resulting geodesic flow from non-parabolic points of $M_0$.  Then, $N_0(u)$ is non singular for each non-parabolic point $\bx(u)$.  If $N_1(\chi(u))$ is 
sufficiently close to $N_0(u)$ then $\widetilde{N}_t(u)$ will be nonsingular.  This is 
given by a $C^1$-condition on the normal vector fields to the surfaces.  If in addition, 
$\theta(u)$, the angle between $\bn_0(u)$ and $\bn_1(\chi(u))$, has small variation as a function of $u$, then the term $\gs(t, \theta)\, \pd{\theta}{\bu}$ will be small in the $C^0$ sense.  Thus, if it is sufficiently small, then together with the $C^1$ closeness of (nonsingular) $N_0(u)$ and $N_1(\chi(u))$ implies that $N_t^{\prime}(u)$ is nonsingular.  Hence, by i) of Theorem \ref{Thm9.2} the flow is uniqely defined.  Together these are $C^2$ conditions on $N_0(u)$ and $N_1(\chi(u))$. \par
\end{Remark}
\begin{proof}[Proof of Theorem \ref{Thm9.2} ]
For 2), given that 1) holds, we may apply ii) of Proposition \ref{Prop5.1}.  
For 3) we may apply Corollary \ref{Cor5.4}.  To prove 1), we will apply i) 
of 
Proposition \ref{Prop5.1}.  We must give a sufficient condition that $N_t(x)$ 
is nonsingular for $0 \leq t \leq 1$.  We choose local coordinates $u$ for a 
neighborhood of $\bx_0$.  
For a geodesic $(\bn_t(u), c_t(u)\bge)$ between $(\bn_0(u), c_0(u)\bge)$ 
and $(\bn_1(u), c_1(u)\bge)$ given by (\ref{Eqn3.4}), we must compute 
$\bn_{t\, u_i}(u)$.  We note that not only $\bn_i,  i = 1, 2$ but also 
$\theta$ 
depends on $u$.  We obtain
\begin{equation}
\label{Eqn9.4}
  \bn_{t\, u_i} \,\, = \,\, \gl(t, \theta)\, \bn_{1\, u_i} \, + \,  \gl(1-t, 
\theta)\, \bn_{0\, u_i}   \, + \, \pd{\gl(t, \theta)}{u_i}\, \bn_1 \, + \,  
\pd{\gl(1-t, \theta)}{u_i}\, \bn_0
\end{equation}
Then, $\pd{\gl(t, \theta)}{u_i} = \pd{\theta}{u_i}\, \pd{\gl(t, 
\theta)}{\theta}$.  Applying (\ref{Eqn9.5}) with $x = t$ and $1-t$, we obtain for the last two terms on the RHS of (\ref{Eqn9.4}) \par
\begin{multline}
\label{Eqn9.6}
\pd{\gl(t, \theta)}{u_i}\, \bn_1 \, + \,  \pd{\gl(1-t, \theta)}{u_i}\, \bn_0   
\,\,  = \,\,   \pd{\theta}{u_i} \big(  \frac{t\, \cos(t\, \theta)}{ 
\sin \theta} \bn_1 \, + \,   \frac{(1 - t)\, \cos((1 - t)\, \theta)}{ 
\sin \theta} \bn_0     \\
  \qquad  - \cot \theta \, (\gl(t, \theta)\, \bn_1 \, + \,  \gl(1-t, \theta)\, 
\bn_0 ) \big)  
\end{multline}
We see that the last expression in (\ref{Eqn9.6}) is a multiple of $\bn_t$.  
We can subtract a multiple of $\bn_t$ from  $\bn_{t\, u_i}$  without 
altering the rank of the matrix $N_t$.  Then, after subtracting 
$\pd{\theta}{u_i} \,\cot \theta \, \bn_t$ from the RHS of (\ref{Eqn9.6}), 
we obtain
\begin{equation}
\label{Eqn9.7}
\pd{\theta}{u_i} \big(  \frac{t\, \cos(t\, \theta)}{ \sin \theta} \bn_1 \, + 
\, 
\frac{(1 - t)\, \cos((1 - t)\, \theta)}{\sin \theta} \bn_0 \big)
\end{equation}
Then, in addition, we can subtract 
$\pd{\theta}{u_i} \, t \cot (t\, \theta) \, \bn_t$ from the RHS of 
(\ref{Eqn9.7}) so the term involving $\bn_1$ is removed.  We are left with
\begin{equation}
\label{Eqn9.8}
\pd{\theta}{u_i} \, \big(  \frac{(1- t)\, \cos((1- t)\,\theta)}{\sin \theta} 
\, -\, t \, \cot (t \theta)\, \frac{\sin ((1-t) \theta)}{\sin \theta} \big)\,  
\bn_0
\end{equation}
Adding the two terms in the parentheses in (\ref{Eqn9.8}), rearranging,  
and using the formula for $\sin (A + B)$, we obtain $\gs (t, \theta)$, so that 
(\ref{Eqn9.8}) becomes $\pd{\theta}{u_i}\, \gs (t, \theta)\, \bn_0$.  Thus, 
applying the preceding to each $\bn_{t\, u_i}$ we may replace each of 
them with 
$$ \gl(t, \theta)\, \bn_{1\, u_i} \, + \,  \gl(1-t, \theta)\, \bn_{0\, u_i}   \, 
+ \, \pd{\theta}{u_i}\, \gs (t, \theta)\, \bn_0  $$
without changing the rank.  We conclude that $N_t$ has the same rank as 
the matrix $N_t^{\prime}$ given in (\ref{Eqn9.2}).  \par
It remains to consider the case when $\theta = 0$.  Then, both 
$\pd{\gl}{\theta}(t, 0) = 0$ and $\pd{\gs}{\theta}(t, 0) = 0$, so that the nonsingularity 
reduces to that for $\tilde N_t(x)$.
\end{proof}
\par
\begin{SimpRemark} \par
If $\bn_1(\chi(x_0)) \neq \bn_0(x_0)$, then there is a neighborhood $x_0 
\in W \subset M_0$ such that $\bn_1(\chi(x)) \neq \bn_0(x)$ for 
$x \in W$.  Then, there is a smooth unit tangent vector field $\bw$ defined 
on $W$ such that $\bn_1(\chi(x))$ lies in the vector space spanned by 
$\bn_0(x)$ and $\bw(x)$, and $\bn_1(\chi(x)) \cdot \bw(x) \geq 0$ for all 
$x \in W$.  Then, smoothness follows explicitly using the geodesics given 
in Proposition \ref{prop4.1} by (\ref{Eqn3.4}).
\end{SimpRemark}

\vspace{1ex}

\section{\bf Results for the Case of Surfaces in $\R^3$}
\label{Sec:surf.case} 
\par
Now we consider the special case of surfaces $M_i \subset \R^3$, i = 1, 2 
for which there is a correspondence given by the diffeomorphism $\chi : 
M_0 \to M_1$.  We suppose each $M_i$ is a generic smooth surface with 
$\bn_0 = (a_1, a_2, a_3)$ and $\bn_1 = (a^{\prime}_1, a^{\prime}_2, 
a^{\prime}_3)$  smooth unit normal vector fields on $M_0$, respectively 
$M_1$.  We assume that $X(u_1, u_2)$ is a local parametrization of $M_0$.  Each 
$a_i$ is a function of  $(u_1, u_2)$ via the local parametrization $X(u_1, u_2)$.  
Likewise, each $a_i^{\prime}$ is a function of  $(u_1, u_2)$ via the local 
parametrization $\chi \circ X(u_1, u_2)$
Also, let $\bn_i(u) \cdot \bx = c_i(u)$ define the tangent planes for $M_0$ 
at $X(u_1, u_2)$, respectively  $M_1$ at $\chi (X(u_1, u_2))$.  \par 
We let
$$ \bn_t \,\, = \,\,  (a_{1 \, t}, a_{2 \, t}, a_{3 \, t}) \,\, = \,\, \gl(t, 
\theta)\, (a^{\prime}_1, a^{\prime}_2, a^{\prime}_3) \, + \,  \gl(1-t, 
\theta)\, (a_1, a_2, a_3)  $$
and $c_t(u) = \gl(t, \theta)\, c_1  \, + \,  \gl(1-t, \theta)\, c_0$.  Then,
\begin{equation}
\label{Eqn10.2} 
N_t \, \, = \,\, 
\begin{pmatrix} 
a_{1 \, t}  &  a_{1 \, t, u_1}  &  a_{1 \, t, u_2}  \\
a_{2 \, t}  &  a_{2 \, t, u_1}  &  a_{2 \, t, u_2}  \\
a_{3 \, t}  &  a_{3 \, t, u_1}  &  a_{3 \, t, u_2} 
\end{pmatrix}
\end{equation}
\begin{SimpRemark} \par
Note here and what follows we use the following notation.  For quantities 
defined for a flow, we denote dependence on $t$ by a subscript.  We also 
want to denote partial derivatives with respect to the parameters $u_i$ 
by a subscript.  To distinguish them, the subscripts appearing after a comma 
will denote the partial derivatives.  Hence, for example, in (\ref{Eqn10.2}) 
$a_{i \, t, u_j} = \pd{a_{i \, t}}{u_j}$. 
\end{SimpRemark}
\subsection*{\it Existence of Envelope Points}  
\par
The sufficient condition that there is a unique point $X_{t_0}(u)$ in the 
Lorentzian geodesic flow in $\R^3$ at time $t = t_0$  is that 
(\ref{Eqn10.2}) 
evaluated at $t = t_0$ and $u = (u_1, u_2)$ is nonsingular.  
Then, the unique point is the solution of the linear system. 
\begin{equation}
\label{Eqn10.4} 
N_{t_0}^T \cdot \bx \, \, = \, \, \bc
\end{equation}
with $\bx$ and $\bc$  column matrices with entries $x_1, x_2, x_3$, 
respectively $c_{t_0}, c_{t_0, u_1},c_{t_0, u_2}$, 
$a_{i \, t, u_j} = \pd{a_{i \, t}}{u_j}$, and $N_{t_0}$ is given by (\ref{Eqn10.2}). \par
Furthermore, the nonsingularity of (\ref{Eqn10.2}) is equivalent to that of 
(\ref{Eqn10.3}).
\begin{equation}
\label{Eqn10.3} 
N^{\prime}_{t_0}\,\, = \,\, \gl(t_0, \theta) \, N_1 \, + \,  \gl(1- t_0, 
\theta) \, N_0  +  \gs (t_0, \theta) \, \pd{\theta}{\bu}\, \bn_0
\end{equation}
where 
\begin{equation}
\label{Eqn10.3a}
\pd{\theta}{\bu}\, \bn_0 \,\, = \,\, 
\begin{pmatrix} 
 0  &  \theta_{u_1} a_{1}  &   \theta_{u_2} a_{1}  \\
 0  & \theta_{u_1}  a_{2}  &  \theta_{u_2} a_{2}  \\
  0 &  \theta_{u_1} a_{3}  &  \theta_{u_2} a_{3} 
\end{pmatrix}
\end{equation}
\par
\subsection*{\it Smoothness of the Envelope}  
For the smoothness of $M_{t_0}$ at the point $X_{t_0}(u_1, u_2)$, we let 
$$  \tilde \bn_{t_0} = (a_{1 \, t_0}, a_{2 \, t_0}, a_{3 \, t_0}, - c_{t_0})$$ 
 evaluated at $u = (u_1, u_2)$.   Also, we let $\tilde \bh_{t_0} = \tilde\bn_{t_0} \times 
\tilde\bn_{t_0\, u_1} \times \tilde\bn_{t_0\, u_1}$, which is the analogue of the 
cross 
product but for vectors in $\R^4$.  It is the vector whose $j$--th entry is 
$(-1)^{j+1}$ times by taking the $3 \times 3$ determinant of the 
submatrix 
obtained by deleting the $j$--th column of 
\begin{equation}
\label{Eqn10.5}
\begin{pmatrix} 
a_{1 \, t_0}  &  a_{2 \, t_0}  & a_{3 \, t_0}  &  -c_{t_0}  \\
a_{1 \, t_0,  u_1}  &  a_{2 \, t_0, u_1}  &  a_{3 \, t_0, u_1}  &  
-c_{t_0, u_1}  \\
a_{1 \, t_0, u_2}  &  a_{2 \, t_0, u_2}  &  a_{3 \, t_0, u_2} &  -c_{t_0, 
u_2} 
\end{pmatrix}
\end{equation}
\par 
Then, we form the $2 \times 2$--matrix $H(\tilde \bn_t(u)) \cdot \tilde \bh_{t})$ 
with $ij$--th entry $\bn_{t, u_i u_j}(u)~\cdot~\tilde \bh_t (u)$ for $u = (u_1, 
u_2)$.  Then, from Theorem \ref{Thm9.2}, we conclude that for a point 
uniquely defined by (\ref{Eqn10.4}) the envelope is smooth at $X_{t_0}(u)$ 
if $H(\tilde \bn_{t_0}(u))\cdot \tilde \bh_{t_0}(u)$ is nonsingular.  \par
\subsection*{\it Envelope Points corresponding to Legendrian Singular 
Points}
Third, the generic Legendrian singularities for surfaces are those given in 
Fig. \ref{fig.1}).  For these: \par
\begin{enumerate} 
\item  At points on cuspidal edges or swallowtail points $z \in \tilde 
M_t$, there is a unique point on $M_t$ which is the unique limit of the envelope 
points corresponding to smooth points of $\tilde M_t$ approaching $z$.
\item  At points $z \in \tilde M_t$ which are tranverse intersections of 
two or three smooth surfaces, or the transverse intersection of a smooth surface
 and a cuspidal edge, there is a unique point in $M_t$ for each smooth surface 
passing through $z$ (and one for the cuspidal edge).  
\end{enumerate}
\par 
\begin{Example} 
\label{Ex10.6}
As an example, we consider the Lorentzian geodesic flows between the surfaces $M_1$ given by $z = 2 - .2 (x^2 + y^2)$, $M_2$ given by $z = .5 - .05 (x^2 + y^2)$, 
and $M_3$ given by $z = 4 - .5 (x^2 + y^2)$.  We consider two correspondences 
and the resulting Lorentzian geodesic flow between them.  The first assigns to 
each point in $M_1$ the point in $M_3$ with the same coordinates $(x, y)$ so the points on the same vertical lines correspond.  For the second, each point in $M_2$ corresponds to the point in $M_3$ in the same vertical line.  For the third, we assign to each point $(x, y, z)$ of $M_2$ the point $(\frac{2}{5} x, \frac{2}{5} y, \frac{2}{5} z + 3.2)$ in $M_3$. \par
Although for the first two there are simple Euclidean geodesic flows in $\R^3$ along the vertical lines, these are not the Lorentzian geodesic flow lines.  
\par
\par
\begin{figure}[ht]
\begin{center} 
\includegraphics[width=6cm]{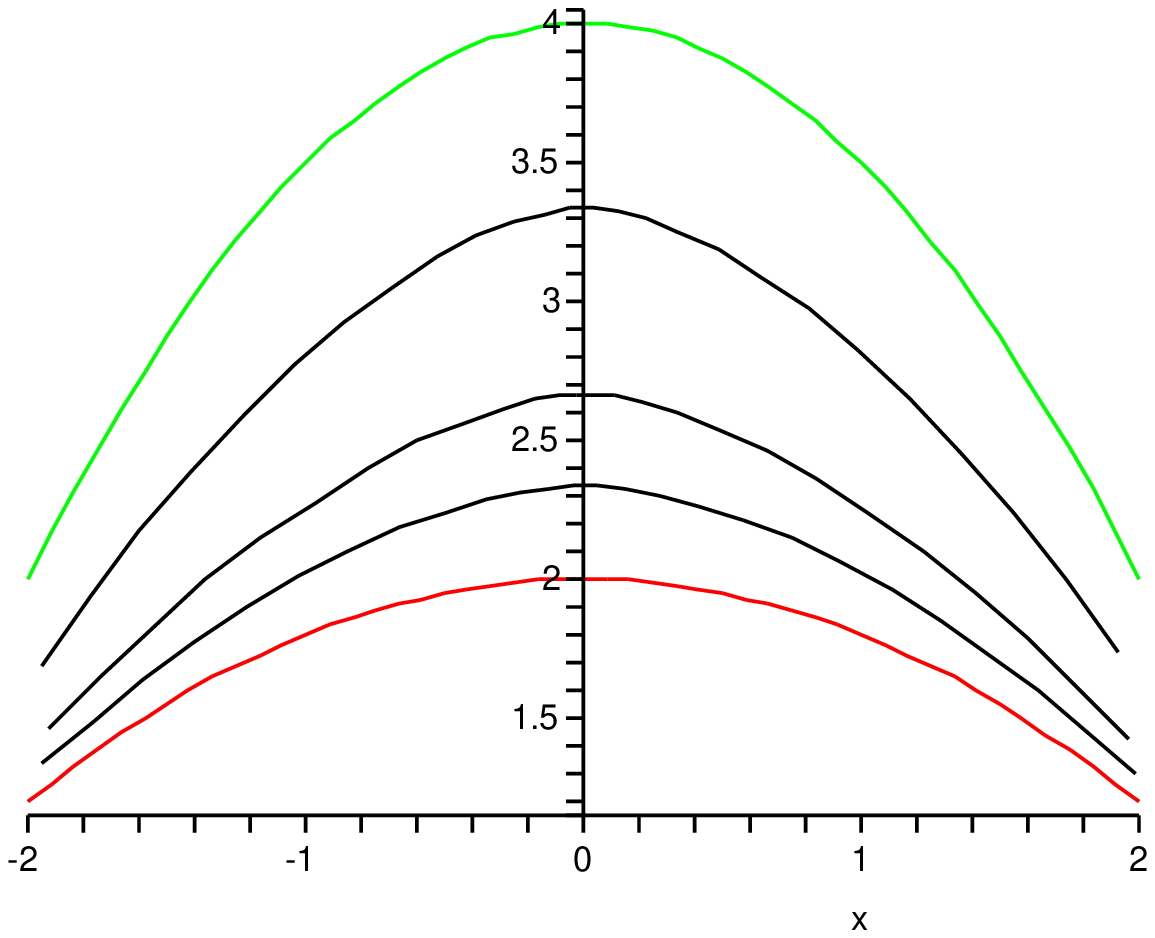}  \hspace{.20in} 
\includegraphics[width=5.5cm]{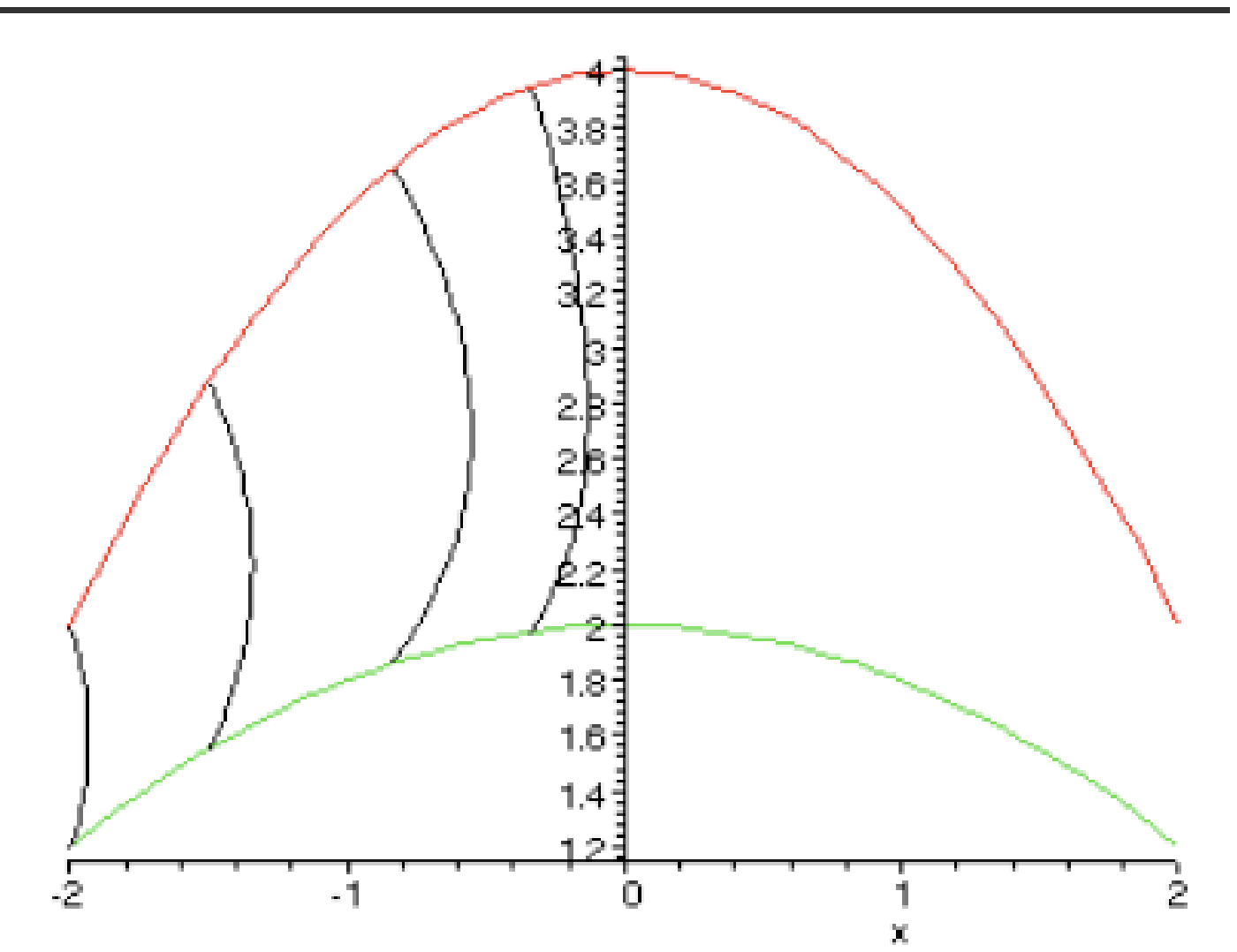}
\end{center}
\hspace{0.4in} (a) \hspace{1.8in} (b) 
 \caption{\label{fig.11} The Lorentzian geodesic flow between $M_1$ and $M_3$ viewed in a vertical plane through the $z$-axis.  In a) are shown the nonsingular level surfaces of the flow and in b) the corresponding geodesic flow curves.  The 
nonsingularity of flow is seen in b) with the geodesic flow curves not intersecting.}   
\end{figure} 
\par 
\begin{figure}[ht]
\begin{center} 
\includegraphics[width=6.0cm]{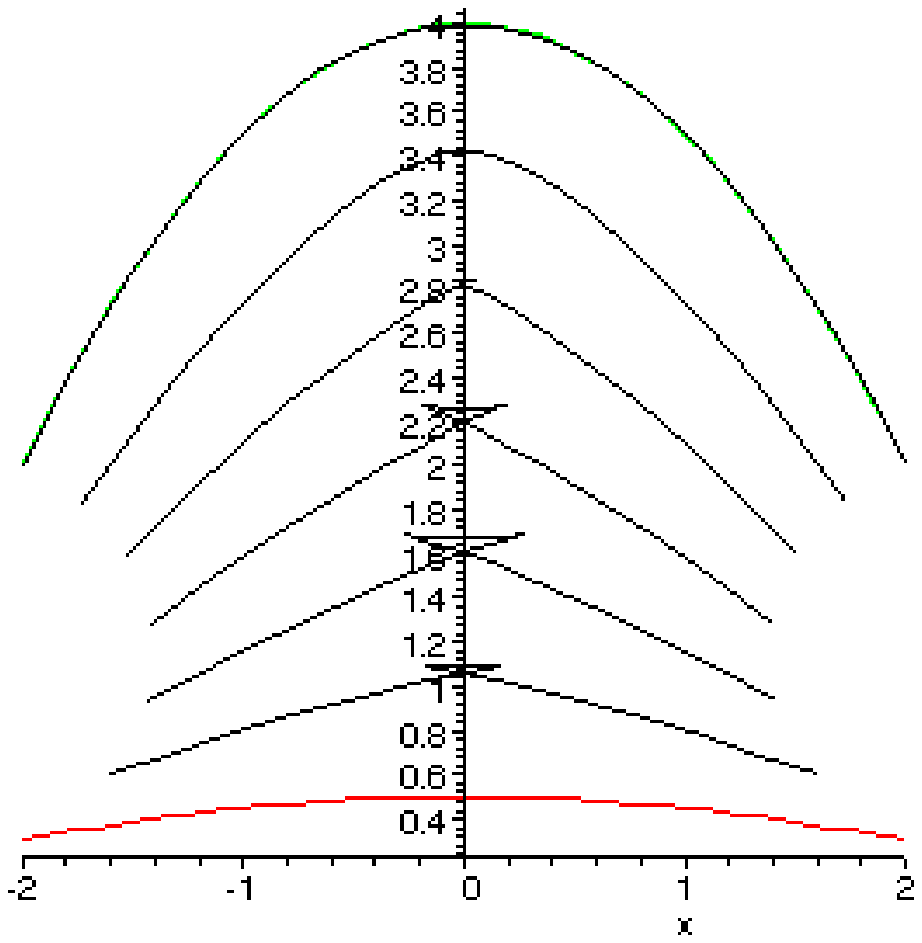}  \hspace{.20in} 
\includegraphics[width=5.5cm]{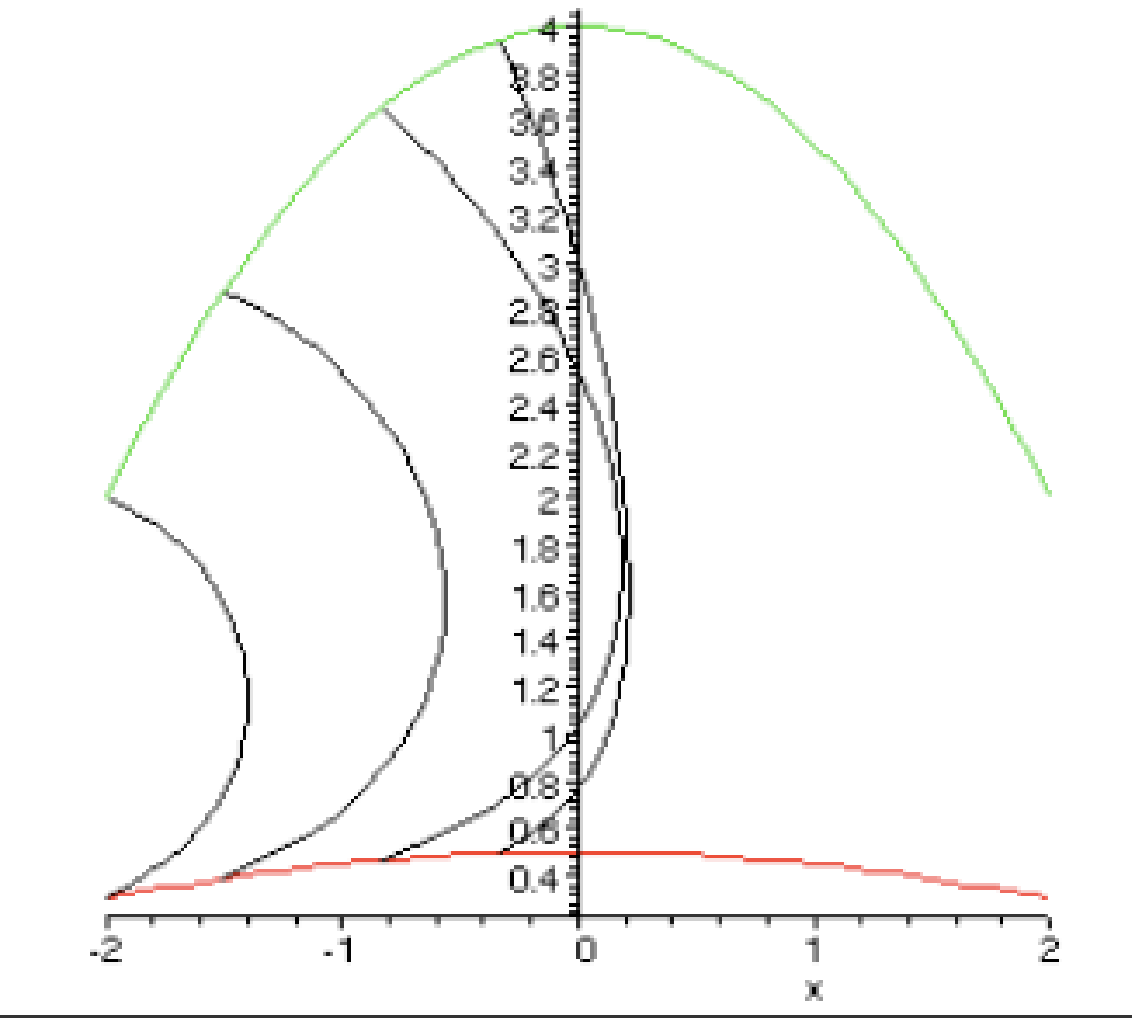} \hspace{.10in}
\end{center}
\hspace{0.4in} (a) \hspace{1.6in} (b)  
 \caption{\label{fig.12} Comparison of Lorentzian geodesic flows between $M_2$ and $M_3$ in a vertical plane through the $z$-axis. 
In a) the level sets exhibit cusp singularity formation.  In b) are shown the Lorentzian 
geodesic curves which intersect and produce the singularities.}
\end{figure} 
\par
For the third, $M_3$ is obtained from $M_2$ by a combination of the homothety of multiplication by $\frac{2}{5}$ combined with the translation by $(0, 0, 3.2)$.  Hence, for the second case the Lorentzian geodesic flow is given by Corollary \ref{Cor5.6} to be along the lines joining the corresponding points and is given by 
$$(x, y, z) \, \mapsto \, ((1- \frac{3}{5} t) x, (1- \frac{3}{5} t) y, (1- \frac{3}{5} t) z + 3.2 t) \, .$$ \par
By the circular symmetry of each surface about the $z$-axis, we may view the 
Lorentzian geodesics in a vertical plane through the $z$-axis.  We may compute both the level sets of the Lorentzian geodesic flow and the corresponding geodesics 
using Proposition \ref{Prop5.1} and solving the systems of equations (\ref{Eqn3.8}).  We show the results of the computations using the software Maple in Figures \ref{fig.11} and \ref{fig.12}.  The Lorentzian geodesic flow between $M_1$ and $M_3$ with the vertical correspondence is nonsingular, as shown by the level sets and geodesic curves in Figure \ref{fig.11}.  
By comparison, the Lorentzian geodesic flow between $M_2$ and $M_3$ for the vertical correspondence is singular.  We see the cusp formation in the level sets in Figure \ref{fig.12} a).  The singularities result from the intersection of the geodesics seen in  and the individual flow curves in b).  We also see that the increased bending of the 
geodesics versus those in Figure \ref{fig.11} result from the increases in the changes in tangent directions, leading to the formation of cusp singularities.  By contrast, for the second correspondence resulting from the action of the element of the extended Poincare group, geodesics are straight lines as shown in Figure \ref{fig.13} and the flow is 
nonsingular.  
\par
\begin{figure}[ht]
\begin{center} 
\includegraphics[width=6.0cm]{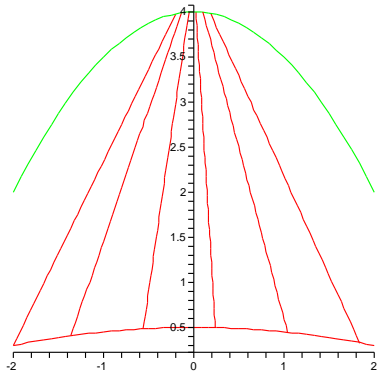}
\end{center}
 \caption{\label{fig.13} Lorentzian geodesic flow between $M_2$ and $M_3$ in a vertical plane through the $z$-axis for the correspondence arising from the action of an element of the extended Poincare group given by a homothety combined with a translation.  The geodesics are straight lines and the flow is nonsingular.}
\end{figure} 
\par
\begin{Remark} For the surfaces, the flows are obtained by rotating the planar figures in Figures \ref{fig.11}, \ref{fig.12} and \ref{fig.13}.  The rotation of Figure \ref{fig.12} yields a circular cusp edge which evolves from a single point on the axis of symmetry.  Hence, the creation point for the cusp singularities does not have generic form. 
\end{Remark}
\end{Example}

\end{document}